\documentclass{amsart}
\usepackage{graphicx}
\usepackage{latexsym}
\usepackage{amsfonts}
\usepackage[all]{xy}
\usepackage{amscd,color}
\usepackage{amsmath}
\usepackage{amssymb}
\usepackage{mathtools} 
\usepackage[pdftex]{hyperref}  
\usepackage[utf8]{inputenc} 
\usepackage[english]{babel}
\usepackage{cancel} 
\usepackage{amssymb, mathrsfs, amsfonts, amsmath}
\usepackage{amsbsy}
\usepackage{amsfonts}
\newtheorem{theoremletter}{Theorem}

\newtheorem{corollary}{Corollary}

\newtheorem{definition}{Definition}
\newtheorem{lemma}{Lemma}
\newtheorem{proposition}{Proposition}
\newtheorem{remark}{Remark}
\newtheorem{theorem}{Theorem}
\newtheorem{example}{Example}
\numberwithin{equation}{section}

\begin{document}
	
	\title[Geometric Inequalities for electrostatic systems]{Geometric inequalities for\\ electrostatic systems with boundary}

	\author{Allan Freitas}
	\address[A. Freitas]{Universidade Federal da Para\'iba - UFPB, Departamento de Matem\'{a}tica, Jo\~ao Pessoa - PB, 58051-900, Brazil} 
	\curraddr{}
	\email{allan@mat.ufpb.br \& allan.freitas@academico.ufpb.br} 
	\thanks{A. Freitas was partially supported by CNPq/Brazil Grant 308141/2025-3.}
	
	\author{Benedito Leandro}
	\address[B. Leandro]{Universidade de Bras\'ilia - UNB, Departamento de Matem\'{a}tica, Bras\'ilia - DF, 70910-900, Brazil} 
	\curraddr{}
	\email{bleandrone@mat.unb.br} 
	\thanks{B. Leandro was partially supported by CNPq/Brazil Grant 303157/2022-4 and FAPDF - 00193-00001678/2024-39.}
	
	\author{Ernani Ribeiro Jr$^{\star}$}
	\address[E. Ribeiro Jr]{Universidade Federal do Cear\'a - UFC, Departamento  de Matem\'atica, Campus do Pici, Av. Humberto Monte, Bloco 914, 60455-760, Fortaleza - CE, Brazil}
	\curraddr{}
	\email{ernani@mat.ufc.br} 
	\thanks{E. Ribeiro Jr was partially supported by CNPq/Brazil Grant 305128/2025-6, CAPES/Brazil and FUNCAP/Brazil ITR-0214-00116.01.00/23.}
	
	\author{Guilherme Sabo}
	\address[G. Sabo]{Universidade de Bras\'ilia - UNB, Departamento de Matem\'{a}tica, Bras\'ilia - DF, 70910-900, Brazil}
	\curraddr{}
	\email{guipaes92@gmail.com}

	\thanks{$^{\star}$Corresponding Author: Ernani Ribeiro Jr (ernani@mat.ufc.br)}


%\date{September 19, 2025}
	
	\dedicatory{}
	
	\begin{abstract}
In this article, we investigate electrostatic systems with a nonzero cosmological constant on compact manifolds with boundary. We establish new geometric pro\-per\-ties for electrostatic manifolds in higher dimensions, extending previous results in the literature. Moreover, we prove sharp boundary estimates and isoperimetric-type inequa\-li\-ties for electrostatic manifolds, as well as volume and boundary inequalities involving the Brown-York and Hawking masses.
	\end{abstract}
	
	\maketitle
	\tableofcontents
	\section{Introduction}

The Einstein-Maxwell equations with cosmological constant $\Lambda$ on a Lorentzian manifold $(\widehat{M}^{n+1},\,\hat{g})$ are given by the following system:
	\begin{equation}
		\left\{\begin{array}{rcll}
			&&{Ric}_{\hat{g}}-\frac{{R}_{\hat{g}}}{2}\hat{g} + \Lambda\hat{g} = 2\left(F\circ F -\frac{1}{4}|F|^2_{\hat{g}}\hat{g}\right);\nonumber\\\\
			&&dF=0\quad\mbox{and}\quad div_{\hat{g}} F=0,\nonumber
		\end{array}\right.
	\end{equation}
	where $F$ stands for the (Faraday) electromagnetic $(0,\,2)$-tensor and $(F\circ F)_{\alpha\beta} = \hat{g}^{\sigma\gamma}F_{\alpha\sigma}F_{\beta\gamma},$ with  Greek indices ranging from $1$ to $n+1$. A static space-time is a product manifold $\widehat{M}^{n+1}=\mathbb{R}\times M^{n},$ endowed with the metric $\hat{g}=-f^2dt^2+g$, where $(M^n,\,g)$ is an oriented $n$-dimensional Riemannian manifold and $f$ is a positive smooth function on $M^n.$ In particular, by choosing the electromagnetic field in the form $F=fE^{\flat}\wedge dt,$ where $E^{\flat}$ denotes the dual $1$-form to the electric field $E$, the Einstein-Maxwell system reduces to a set of equations adapted to the electrostatic setting (cf. \cite{chrusciel2017non,tiarlos,leandro2024}).
	
	\begin{definition}\label{def1}
		The Einstein-Maxwell equations with cosmological constant $\Lambda$ for the electrostatic space-time associated to $(M^n,\, g,\, f,\,E)$ are given by
\begin{equation}\label{eq1}
\left\{
\begin{aligned}
\nabla^2 f &= f \left( Ric - \dfrac{2}{n-1} \Lambda g + 2 E^\flat \otimes E^\flat - \dfrac{2}{n-1} |E|^2 g \right), \\\\
\Delta f &= \dfrac{2}{n-1} \left( (n-2)|E|^2 - \Lambda \right) f, \\\\
0 &= div(E) \quad \text{and} \quad 0 = d(f E^\flat).
\end{aligned}
\right.
\end{equation}
We say that $(M^n,\, g,\, f,\,E)$ is an electrostatic system if the system of equations \eqref{eq1} is satisfied for some constant $\Lambda\in\mathbb{R}.$ In particular, if $(M^n,\,g)$ is complete, then the electrostatic system is also complete.
	\end{definition}
	
We remark that \( Ric \), \( \nabla^2 \), \(div\), and \( \Delta \) denote the Ricci tensor, the Hessian tensor, the divergence, and the Laplacian with respect to the metric \( g \), respectively. The smooth function $f$ is referred to as the {\it lapse function} (or electrostatic potential), and $E$ denotes the electric field. If $M$ has a horizon boundary $\partial M$, we additionally assume that $f^{-1}(0)=\partial M$ (cf. \cite{cederbaum2016uniqueness, chrusciel2017non,tiarlos}). An electrostatic system with zero cosmological constant is referred to as an electro\-vacuum system. In the absence of an electric field $E,$ the system reduces to a vacuum static system. Among the explicit examples of electrostatic spacetimes, we highlight the vacuum static models (namely, the de Sitter and Nariai systems, see Example \ref{Exa1}) as well as the Reissner-Nordstr\"om-de Sitter (RNdS), Majumdar-Papapetrou (MP), charged Nariai, cold black hole and ultracold black hole solutions; further details are provided in Section 2.2.

To proceed, we recall that the charge \( Q \) of a hypersurface \( \Sigma \) in \( (M^n, g) \) is given by
\begin{equation} \label{charge}
Q(\Sigma) = \frac{1}{\omega_{n-1}} \sqrt{\frac{2}{(n-1)(n-2)}} \int_{\Sigma} \langle E, \nu \rangle \, dA_g,
\end{equation}
where \( \omega_{n-1} \) denotes the area of the standard unit \((n-1)\)-sphere, and \( \nu \) is the unit normal vector to \( \Sigma \). Moreover, the classification problem for electro-vacuum spacetimes can be formulated as follows: assume that $Q_iQ_j \geq 0,$ for all $\,i,\,j,$ where $Q_i$
denotes the charge of the $i$-th connected degenerate component of an electrically charged black hole. Under this condition, the black hole must be either a Reissner–Nordstr\"om (RN) or a Majumdar-Papapetrou (MP) black hole. Several important contributions to the classification of electro-vacuum spacetimes can be found in the literature; see, for example, \cite{cederbaum2016uniqueness,chrusciel1999,chrusciel2007, kunduri2018, leandro2023, Lucietti}, and the references therein. In particular, by the positive mass theorem, any asymptotically flat electro-vacuum spacetime is conformally flat. That is, the Cotton tensor $C$ vanishes in three dimensions, and the Weyl tensor $W$ vanishes in higher dimensions. This, in turn, implies that the only admissible solutions to the electro-vacuum system under such conditions are precisely the Reissner-Nordstr\"om and Majumdar–Papapetrou spacetimes (cf. \cite[Theorem 3.6]{chrusciel1999}).

In this article, we investigate whether the de Sitter system is the unique compact, simply connected electrostatic system with positive scalar curvature. This question, originally posed in the context of vacuum static spaces, has been extensively studied within general relativity and is closely related to the {\it Cosmic No-Hair Conjecture}, formulated by Boucher, Gibbons, and Horowitz \cite{boucher1} (see also \cite{boucher}):  ``{\it The only compact vacuum static system \((M^n, g, f)\) with positive scalar curvature and connected boundary is the de Sitter system, with static potential \(f\) given by the height function}''. While several results have confirmed the conjecture under additional assumptions (see, for example, \cite{ambrozio,bouch,boucher1,fried,HY,koba, laf,reilly}), counterexamples were constructed by Gibbons, Hartnoll and Pope \cite{ghp} in dimensions \( 4 \leq n \leq 8 \), and a simply connected counterexample in all dimensions \( n \geq 4 \) was later presented by Costa, Di\'ogenes, Pinheiro e Ribeiro \cite{costa}. To the best of our knowledge, the conjecture is still open for dimension \( n = 3 \).

A fruitful approach to investigating rigidity phenomena is through the establishment of obstruction results, such as geometric inequalities, which serve to obtain classification results as well as rule out some potential new examples. In this context, a distinguished result, by Boucher, Gibbons and Horowitz \cite{boucher1} and Shen \cite{sh}, asserts that the de Sitter system maximizes the boundary area among all compact vacuum static spaces with positive scalar curvature and connected boundary. More precisely, they proved the following:

\begin{quote}
\textit{Let \((M^3, g, f)\) be a compact, oriented vacuum static space with connected boundary and scalar curvature equal to \(6\). Then the area of \(\partial M\) satisfies the inequality
\[
|\partial M| \leq 4\pi.
\]
Moreover, equality holds if and only if \((M^3, g)\) is isometric to the de Sitter system.}
\end{quote} In recent years, also motivated by the classical isoperimetric inequality, the study of boun\-dary and volume estimates for special classes of manifolds (or metrics) has made significant advances. Such estimates have been established for static spaces in, e.g., \cite{ambrozio,BMaz,costa,HMR}, as well as for critical metrics of the volume functional in, e.g., \cite{BDR1,BLF,BSilva,Batista,CEM,FY,Y2023}; see also \cite{ernani,ernani2}. Important related inequalities involving charge can be found in \cite{Jaracz,MKhuri}. In this spirit, by employing the recent genera\-lized Reilly formula due to Qiu and Xia \cite{qiu} (see also \cite{LiXia}), we establish a sharp boundary estimate for compact electrostatic manifolds. More precisely, we obtain the following result.

\begin{theoremletter}
\label{main}
Let $(M^n,\, g,\, f,\,E)$ be a compact electrostatic system with connected boundary $\partial M$ satisfying $|E|^2\leq\frac{\Lambda}{(n-2)}.$ Then we have:
\begin{equation}
\label{plmnbv5}
\sqrt{\frac{\alpha(n+2)}{2n\beta^{2}}}|\partial M|\leq  Vol(M),
\end{equation}
where $ \alpha = \min_{M}\dfrac{2}{n - 1}\left(\Lambda - (n - 2)|E|^2\right)$ and $ \beta = \max_{M}\dfrac{2}{n - 1}\left(\Lambda - (n - 2)|E|^2\right)$. Moreover, equality holds in (\ref{plmnbv5}) if and only if $(M^n,\, g,\, f)$ is isometric to the de Sitter system.
\end{theoremletter}

\begin{remark}
We remark that the conclusion of Theorem \ref{main} remains valid even when the boundary is disconnected, provided that the surface gravities \( \kappa_i := |\nabla f| \big|_{\Sigma_i} \), where \( \Sigma_i \) denotes the connected components of \( \partial M \), are all equal to the same constant (see Theorem~\ref{teoAgen}).
\end{remark}

\begin{remark}
\label{remNew}
From the physical viewpoint, the pointwise assumption $|E|^2 \leq \frac{\Lambda}{n-2}$ requires the electromagnetic field strength to be extremely small, far smaller than any electric field that arises in natural settings. In particular, the inequality forces the electromagnetic energy density to be no larger than the energy density associated with the cosmological constant. However, the latter is known to be very small, whereas ordinary electromagnetic fields, even exceptionally weak ones, produce energy densities many orders of magnitude greater; see \cite{Barrow}. Consequently, this condition should be regarded as a mathematically convenient geometric restriction rather than a physically realistic assumption.
\end{remark}

Our next result may be viewed as an analogue of a Chru\'sciel-type inequality in the setting of electrostatic systems. For motivating examples and a broader context in which such inequalities arise, we refer the reader to \cite[p.~13]{Crus} (see also \cite{HMR}). 

To establish our result, we impose the condition that the electric field \( E \) is  parallel to the gradient of the lapse function \( f \). This assumption appears to be necessary even in the three-dimensional case treated by Cruz, Lima, and Sousa in \cite[Theorem E]{tiarlos}; see Remark \ref{remarkNEW} and Lemma~\ref{theo0} in Section \ref{Sec2}. In particular, the next result addresses this missing step in their argument. Specifically, we prove the following sharp inequality.

\begin{theoremletter}
\label{Crusc_th}
		Let $(M^n,\, g,\, f,\, E)$ be a compact electrostatic system such that $E$ is parallel to $\nabla f$ and $\partial M=\cup_{i=1}^{l} \Sigma_{i}$, where $\Sigma_{i}$ are the connected components of $\partial M.$ Suppose that $\vert E\vert^{2}< \frac{\Lambda}{(n-2)}.$ Then we have:
		\begin{equation}\label{Crusc_eq}
			\frac{1}{2}(n-1)(n-2)\omega^{2}_{n-1}\sum_{i=1}^{l}  \frac{\kappa_{i}Q(\Sigma_i)^2}{\vert\Sigma_i\vert} +\frac{n-2}{n}\Lambda\sum_{i=1}^{l} \kappa_{i}|\Sigma_i|\leq\sum_{i=1}^{l} \kappa_{i} \int_{\Sigma_{i}}  \frac{R^{\Sigma_{i}}}{2}dA_g ,
		\end{equation}
		where $\kappa_{i}=|\nabla f|\Big|_{\Sigma_{i}}$ and $\omega_{n-1}$ is the area of the standard $(n-1)$-sphere. Here, $R^{\Sigma_{i}}$ stands for the scalar curvature of $\Sigma_{i}$. Moreover, equality holds in \eqref{Crusc_eq} if and only if $(M^n,\,g,\,f)$ is isometric to the de Sitter system.  
\end{theoremletter}

\begin{remark}
We emphasize that Theorem \ref{Crusc_th} can be regarded as a higher-dimensional gene\-ra\-li\-zation of \cite[Theorem~E]{tiarlos}. The condition that \(E\) is parallel to \(\nabla f\) always holds along \(\partial M\); see assertion $(iv)$ in Proposition \ref{properties}. Moreover, this property is satisfied in all known explicit examples of electrostatic systems, as discussed in Section \ref{Sec2}.
\end{remark}

\begin{remark}\label{rem_n=3}
The case \(n = 3\) is of particular interest. In this specific dimension, the boundary \(\partial M = \bigcup_{i=1}^{\ell} \Sigma_{i}\) consists of closed surfaces, and the scalar curvature of each component satisfies \(R^{\Sigma_i} = 2K\), where \(K\) denotes the Gauss curvature. Consequently, by the Gauss–Bonnet theorem, we obtain the following area estimate:

\begin{equation*}
\sum_{i=1}^{\ell} k_i \left( \dfrac{16\pi^2 Q(\Sigma_i)^2}{|\Sigma_i|} + \dfrac{\Lambda}{3} |\Sigma_i| \right) \leq 4\pi \sum_{i=1}^{\ell} k_i,
\end{equation*} which coincides precisely with \cite[Theorem E]{tiarlos}. Furthermore, if the boundary is assumed to be Einstein, we derive a related boundary estimate; see Theorem~\ref{alan1}. For a significant topological characterization in the five-dimensional setting, see Corollary~\ref{coro2alan}.
\end{remark}

In the sequel, we relax the assumption that 
\( E \) is parallel to \( \nabla f \), which was required in the previous theorem. To this end, however, we must impose a control  on the electric field. Our approach is based on an integral identity for electrostatic systems, which is of independent interest; see Proposition~\ref{propprop}. Although an analogous result holds in higher dimensions (see Theorem~\ref{geometricineq}), we state the three-dimensional case separately to allow a direct comparison with \cite[Theorem~D]{tiarlos}.

	\begin{theoremletter}\label{livreE}
	Let \( (M^{3},\,g,\,f,\,E) \) be a compact electrostatic system with connected boundary \( \partial M \) and positive cosmological constant. Suppose that
	\[
	\vert E\vert^{2} <  \frac{\sqrt{5}}{5} \Lambda .
	\]
Then $\partial M$ is a $2$-sphere and there exists a positive constant $c$ such that
	\begin{equation}
		\label{eqkl901p}
		\kappa c\vert\partial M\vert \leq 4\pi\Lambda,
	\end{equation} where \( \kappa = |\nabla f| \big|_{\partial M} \). Moreover, equality holds in (\ref{eqkl901p}) if and only if \( (M^{3},\,g,\,f) \) is isometric to the de Sitter system with $c=\Lambda$ and $\kappa = \dfrac{\Lambda}{3}$.
\end{theoremletter}

\begin{remark}\label{rem_E}
This bound on the norm of \( E \) considered in Theorem \ref{livreE} arises naturally, even in the case where the scalar curvature is constant. In particular, when \( n = 3 \), it is known that the electric field satisfies \( |E|^{2} < \Lambda \), and the scalar curvature satisfies \( R < 4\Lambda \); see assertion~\((v)\) of Proposition~\ref{properties} in Section \ref{Sec2}.
\end{remark}

By a different approach, we derive a boundary estimate for electrostatic manifolds that depends only on the cosmological constant $\Lambda,$ that is, the estimate is independent of the surface gravity \( \kappa = |\nabla f| \big|_{\partial M} \). As before, we present here the result in the three-dimensional setting, although it arises as a special case of a more general boundary estimate valid in higher dimensions (see Theorem~\ref{teorigidezint}).

\begin{theoremletter}\label{teorigidezint1}
		Let $(M^3,\, g,\, f,\, E)$ be a compact electrostatic system with connected boundary $\partial M$. Suppose that $$\vert E\vert^{2}< \dfrac{\Lambda}{5}.$$ Then $\partial M$ is a $2$-sphere and
		\begin{equation}
		\label{eqjkl12340}
			\Lambda|\partial M|\leq12\pi ,
		\end{equation}
		Moreover, equality holds in (\ref{eqjkl12340}) if and only if $(M^3,\,g,\,f)$ is isometric to the de Sitter system.
	\end{theoremletter}

In the second part of this work, we focus on the role of quasi-local masses in the setting of electrostatic systems. In Newtonian gravity, the mass of a region can be naturally defined by integrating a mass density function. However, in general relativity, the situation is considerably more subtle due to the Equivalence Principle, which precludes the existence of a local energy density for the gravitational field. The definition of mass thus becomes a central and subtle problem in general relativity. In 1982, Penrose \cite{penrose1} identified several major open questions in the field, placing at the forefront the challenge of formulating a suitable quasi-local definition of energy-momentum. Mathematically, the formulation of this problem presents inherent challenges, and as a result, several definitions of quasi-local mass have been proposed over the years in an attempt to capture a coherent and physically meaningful notion of mass in general relativity. For a detailed discussion of the challenges and desirable properties associated with such definitions, we refer the reader to \cite{alaee2023, christodoulou1986} and the references therein.

In order to proceed, we recall the definition of the  Brown-York mass. Let \(\Sigma\) be a connected hypersurface in \((M^n, g)\) such that \((\Sigma, g|_{\Sigma})\) can be embedded in \(\mathbb{R}^n\) as a convex hypersurface. Then, the Riemannian Brown-York mass \(\mathfrak{M}_{BY}\) of \(\Sigma\) with respect to \(g\) is given by
\[
\mathfrak{M}_{BY}(\Sigma, g) = \int_{\Sigma} (H_0 - H) \, dA_g,
\]
where \(H_0\) and \(H\) denote the mean curvatures of \(\Sigma\) as a hypersurface of \(\mathbb{R}^n\) and \(M^n\), respectively, and \(dA_g\) is the volume element on \(\Sigma\) induced by \(g\). In \cite{yuan}, Yuan proved a boundary estimate
for vacuum static spaces in terms of the Riemannian Brown-York mass. Similar estimates were obtained in \cite{costa,ernani}. Motivated by these results, we have the following sharp boundary estimate involving the Riemannian Brown-York for electrostatic manifolds.

\begin{theoremletter}
\label{teoBY}
Let $(M^n,\,g,\,f,\,E)$ be a compact electrostatic system with (possibly disconnected) boundary $\partial M$ and $\Lambda+|E|^{2}>0$. Suppose that each boundary component $(\Sigma_{i},\,g)$ can be isometrically embedded in $\mathbb{R}^n$ as a convex hypersurface. Then we have
\begin{eqnarray}\label{eqq4}
|\Sigma_{i}|\leq c\, \mathfrak{M}_{_{BY}}(\Sigma_i,g),
\end{eqnarray} where $c$ is a positive constant. Moreover, equality holds for some component $\Sigma_{i}$ if and only if $(M^n,\, g,\, f)$ is isometric to the de Sitter system. 
\end{theoremletter}

We note that, by virtue of the solution to the Weyl problem, the isometric embedding condition in Theorem \ref{teoBY} can be replaced by appropriate curvature bounds, as for instance, by requiring positive Gaussian curvature when $n = 3,$ see, e.g. \cite{FY,yuan}.

Another important notion of quasi-local mass, the Hawking mass, has received considerable attention due to its essential role in the proof of the Riemannian Penrose inequality \cite{huisken2001}. The standard Hawking mass of a surface \(\Sigma^2\) in a given Riemannian manifold \((M^3, g)\) is defined by

\[
\mathfrak{M}_{H}(\Sigma) = \sqrt{\frac{|\Sigma|}{16 \pi}} \left( 1 - \frac{1}{16\pi} \int_{\Sigma} H^2 \, dA_{g} \right),
\]
where \(|\Sigma|\) and \(H\) denote the area of the surface and its mean curvature with respect to the metric \(g\), respectively. Notice that if \(\Sigma\) is a minimal surface, then the Hawking mass is positive. The Hawking mass is often utilized as a lower bound for the Bartnik quasi-local mass (see \cite{miao2020}). Christodoulou and Yau \cite{christodoulou1986} demonstrated that the Hawking mass is non-negative for stable constant mean curvature (CMC) spheres in $3$-manifolds with non-negative scalar curvature. This result highlights the delicate nature of the Hawking mass, even regarding its positivity. In the presence of an electric field, the notion of charged Hawking mass becomes relevant and is defined below. We emphasize that this appropriate definition was first introduced in \cite[Section~3]{khuri} for the case  \(\Lambda = 0\) (see also \cite{hayward}), and was later extended by \cite{baltazar2023} to the more general setting of an arbitrary cosmological constant.

\begin{definition}\label{hawkingmass}
Let $(M^3,\,g)$ be a three-dimensional Riemannian manifold and $\Sigma^2$ a closed surface on $M^3$. The charged Hawking mass is defined by
\begin{eqnarray*}
\mathfrak{M}_{CH}(\Sigma) = \sqrt{\frac{|\Sigma|}{16 \pi}}\left(\frac{1}{2}\chi(\Sigma) - \frac{1}{16\pi}\int_{\Sigma}(H^2 + \frac{4}{3}\Lambda)dA_{g} + \frac{4\pi}{|\Sigma|}Q(\Sigma)^2 \right),
\end{eqnarray*}
where $|\Sigma|$ stands for the area of $\Sigma$. Here, $\chi(\Sigma)=2(1-g(\Sigma))$ is the Euler characteristic and $g(\Sigma)$ is the genus of $\Sigma.$
\end{definition}

Very recently, the second and fourth named authors \cite{leandro2025} established sharp lower bounds for the charged Hawking mass of constant mean curvature surfaces with index zero or one in electrostatic systems. The result presented below consists of Riemannian Penrose-type inequalities (cf. \cite{huisken2001}) that relate the charged Hawking mass of the boundary to the total charge and the boundary area; see also related inequalities in \cite{Jaracz,MKhuri}. In particular, the assumption that the electric field $E$ is parallel to $\nabla f$ plays a crucial role in deriving one of the inequalities. To be precise, we get the following result.

\begin{theoremletter}\label{hawking2}
Let \((M^3, g, f, E)\) be a compact electrostatic system with connected boundary \(\partial M\). Then the following assertions hold:
\begin{itemize}
    \item[(i)] If \(E\) is parallel to \(\nabla f\) and \(|E|^{2} < \Lambda\), then
    \begin{equation}\label{h1}
        \mathfrak{M}_{CH}(\partial M) \geq \left( \frac{8\pi}{|\partial M|} Q(\partial M)^2 \right) \sqrt{\frac{|\partial M|}{16\pi}}.
    \end{equation}
    
    \item[(ii)] If \(|E|^{2} < \dfrac{\Lambda}{5}\), then
    \begin{equation}\label{h2}
        \mathfrak{M}_{CH}(\partial M) \geq \left( \frac{4\pi}{|\partial M|} Q(\partial M)^2 \right) \sqrt{\frac{|\partial M|}{16\pi}}.
    \end{equation}
\end{itemize} Moreover, equality holds in \eqref{h1} or \eqref{h2} if and only if \(E = 0\) and \((M^3,\, g,\, f)\) is isometric to the de Sitter system.
\end{theoremletter}

\vspace{0.30cm}

The remainder of this article is organized as follows. In Section 2, we review fundamental properties of electrostatic systems, establish preliminary results, and recall key theorems that will be used throughout the paper. Moreover, it contains some explicit examples of electrostatic systems. In Section 3, we present our results on boundary area estimates and rigidity, including the proofs of Theorems~\ref{main}, \ref{Crusc_th}, \ref{livreE}, and \ref{teorigidezint1}. Finally, in Section 4, we address results concerning the Brown–York and Hawking masses, specifically Theorems~\ref{teoBY} and \ref{hawking2}.

	\section{Background}
	\label{Sec2}
	
	Throughout this section, we recall some basic results that will be used in the text and establish some preliminary results. Moreover, we present some examples of electrostatic systems in higher dimensions. 

	\subsection{Preliminary results} We begin by collecting some properties of the electrostatic system 
\begin{equation}\label{s1}
\left\{
\begin{aligned}
\nabla^2 f &= f \left( Ric - \dfrac{2}{n-1} \Lambda g + 2 E^\flat \otimes E^\flat - \dfrac{2}{n-1} |E|^2 g \right), \\\\
\Delta f &= \dfrac{2}{n-1} \left( (n-2)|E|^2 - \Lambda \right) f, \\\\
0 &= div(E) \quad \text{and} \quad 0 = d(f E^\flat).
\end{aligned}
\right.
\end{equation} Although some of these features are already known (see, for example, \cite[Lemma 4]{tiarlos}), we include them here for the sake of completeness.

	\begin{proposition}\label{properties}
		Let $(M^n,\,g,\,f,\,E)$ be an electrostatic system with a non-empty boundary. Then the following assertions hold:
		
		\begin{itemize}
			\item[(i)] The scalar curvature of $(M^{n},\,g)$ is given by
			\begin{equation}
			\label{rrr}
			R=2\Lambda+2|E|^{2};
			\end{equation}
			\item[(ii)] The boundary \( \partial M \) is totally geodesic. In particular,
\begin{equation} \label{gausseq}
\mathring{Ric}(\nu, \nu) = \frac{n-2}{2n} R - \frac{1}{2} R^{\partial M},
\end{equation}
where \( \mathring{Ric} = Ric - \frac{R}{n}g \) is the traceless Ricci tensor, and \( R^{\partial M} \) is the scalar curvature of \( \partial M \).

			\item[(iii)] If $\partial M=\displaystyle\cup_{i=1}^{l} \Sigma_{i}$, where $\Sigma_{i}$ are the connected components of $\partial M$, then $\kappa_{i}:=|\nabla f|\Big|_{\Sigma_{i}}$ are non-null constant (they are called surface gravities);
			\item[(iv)] $\nabla f$ and $E$ are proportional along $\partial M$. In particular, $|\langle E,\nabla f\rangle|=|E||\nabla f|$ along $\partial M$;			
			\item[(v)] If $M$ is compact and $\vert E\vert$ constant, then $$\Lambda> \left(n-2\right)|E|^{2}.$$
			In particular, 
			\[
2\Lambda \leq R < \frac{2(n-1)}{n-2} \Lambda;
\]
			\item[(vi)] $(M^n, g, f)$ is sub-static, i.e.,
			\begin{equation}
				fRic-\nabla^2 f+(\Delta f)g\geq 0. \label{sub_static}
			\end{equation}
		\end{itemize}
	\end{proposition}

	\begin{proof}
The proof is structured according to each assertion:
	
		\begin{itemize}
		\item[(i)] We take the trace of the first equation in \eqref{s1} in order to infer $$\Delta f= f\left(R -\frac{2n}{n-1}\Lambda -\frac{2}{n-1}|E|^2\right)=\left(\frac{2(n-2)}{n-1}|E|^2 -\frac{2}{n-1}\Lambda\right)f,$$ which combined with the second equation in \eqref{s1} proves the stated identity. 
		
		\item[(ii)] First, we observe that \( |\nabla f(x)| \neq 0 \) for all \( x \in \partial M \). Indeed, consider a unit-speed geodesic \( \sigma: [0,1] \to M^n \) such that \( \sigma(0) = x \in \partial M \) and \( \sigma(1) \in \operatorname{int}(M) \). Define the function \( \theta: [0,1] \to \mathbb{R} \) by \( \theta(t) = f(\sigma(t)) \). Then \( \theta(0) = f(x) = 0 \), and
\begin{eqnarray*}
\theta''(t) = \nabla^2 f(\sigma'(t), \sigma'(t)) = A(t)\theta(t),
\end{eqnarray*} for some function \( A(t) \). If \( \nabla f(x) = 0 \), then \( \theta(t) \) would satisfy the initial value problem \( \theta''(t) = A(t)\theta(t) \) with \(  \theta(0)=\theta'(0) = 0 \), implying that \( \theta \) is identically zero near \( x \), and thus, \( f \) would vanish along \( \sigma \), leading to a contradiction.

Since \( \partial M = f^{-1}(0) \), we may take \( \nu = -\frac{\nabla f}{|\nabla f|} \) as a unit normal vector field on \( \partial M \). Hence, for any vector fields \( X, Y \in \mathfrak{X}(\partial M) \), the second fundamental form \( A \) of \( \partial M \) satisfies
\begin{equation*} \label{Cal-A}
A(X,Y) = -\langle \nabla_{X} \nu, Y \rangle = \frac{1}{|\nabla f|} \nabla^{2} f(X,Y) = 0,
\end{equation*}
along \( \partial M \). The Gauss equation then yields identity \eqref{gausseq} directly.

		\item[(iii)] For any $X\in\mathfrak{X}(\partial M),$ we compute
		 \begin{eqnarray*}
		X(|\nabla f|^2)=2\langle\nabla_X\nabla f,\nabla f\rangle=2\nabla^2f(X,\nabla f)=0,
	\end{eqnarray*}
	along $\partial M$.
		
		\item[(iv)] Since
		$$0=d(fE^{\flat})=df\wedge E^{\flat}+fdE^{\flat},$$
		and $f=0$ on $\partial M,$ it follows that $df\wedge E^{\flat}=0$ on $\partial M$.
		\item[(v)] Integrating the second equation in  \eqref{s1} and using the notation from item (iii), we obtain
\begin{eqnarray*}
-\sum_{i} \kappa_{i} |\Sigma_{i}| = \int_{M} \Delta f \, dV_{g} = \frac{2}{n-1} \left( (n-2)|E|^{2} - \Lambda \right) \int_{M} f \, dV_{g}.
\end{eqnarray*} Since \( f > 0 \) in \( M^n \), the result follows.

		\item[(vi)] A direct computation using system \eqref{s1} gives
\[
f \, Ric - \nabla^2 f + (\Delta f) g = 2f \left( |E|^{2}g - E^{\flat} \otimes E^{\flat} \right),
\]
and the desired conclusion follows from the Cauchy–Schwarz inequality.
\end{itemize} This finishes the proof.
	\end{proof}

	\begin{remark}	
It is worth noting that the condition \( d(fE^{\flat}) = 0 \) is automatically satisfied when \( fE = \nabla \psi \) for some smooth (electric) potential function \( \psi \) defined on \( M \). Moreover, by Poincar\'e's lemma (see, for instance, \cite[Corollary 11.50]{lee}), it follows that, locally,
\begin{eqnarray*}
fE = \nabla \psi,
\end{eqnarray*}
where \( \psi: U \subseteq M \to \mathbb{R} \) is a smooth function defined on some open subset \( U \subseteq M \). In fact, if \( M \) is contractible, the potential \( \psi \) can be defined globally.
\end{remark}

\begin{remark}
\label{sub_sta}
We also highlight that the condition (vi) established in Proposition~\ref{properties} carries an important physical interpretation related to Einstein's field equations. Consider an $(n+1)$-dimensional static spacetime $(\widehat{M}, \hat{g})$, where
\begin{equation}\label{static_space}
    \widehat{M} = \mathbb{R} \times M \, , \qquad \hat{g} = -f^2 \, \mathrm{d}t \otimes \mathrm{d}t + g,
\end{equation} $(M^n, g)$ is a Riemannian manifold and $f \in C^\infty(M)$ is a positive smooth function. Suppose this spacetime satisfies the Einstein field equations:
\begin{equation}\label{Einst_eq}
    Ric_{\hat{g}} + \left( \Lambda - \frac{1}{2} R_{\hat{g}} \right) \hat{g} = \mathfrak{T},
\end{equation} where $\mathfrak{T}$ stands for the stress-energy tensor and $\Lambda$ is the cosmological constant.

The spacetime $(\widehat{M}, \hat{g})$ is said to satisfy the \emph{null energy condition (NEC)} for \eqref{Einst_eq} if
\[
\mathfrak{T}(Y, Y) \geq 0
\]
for all null vectors $Y$, that is, vectors satisfying $\hat{g}(Y, Y) = 0$. This condition is a natural geometric assumption and plays a central role in Penrose’s singularity theorem~\cite{penrose}. It is also closely related to the sub-static condition \eqref{sub_static} (see, e.g.,~\cite[Lemma~3.8]{wang} and~\cite[Section~2]{cfmr}).

Indeed, up to a rescaling, a null vector $Y$ can be expressed as $Y = \widehat{e_0} + X$, where $\widehat{e_0} = \frac{1}{f} \frac{\partial}{\partial t}$ and $X \in \mathfrak{X}(M)$ satisfies $g(X, X) = 1$. Given the form of the metric in \eqref{static_space} and the fact that $\hat{g}(\widehat{e_0}, \widehat{e_0}) = -1$, applying the (NEC) to $Y$ yields:
\begin{eqnarray*}
    0 \leq \mathfrak{T}(Y, Y) &=& \mathfrak{T}_{00} + \mathfrak{T}_{ij} X^i X^j \\
    &=& \left( -\Lambda + \frac{R}{2} \right) + \left[ Ric - \frac{\nabla^2 f}{f} + \frac{\Delta f}{f} g \right] (X, X) + \left( \Lambda - \frac{R}{2} \right) g(X, X) \\
    &=& \left[ Ric - \frac{\nabla^2 f}{f} + \frac{\Delta f}{f} g \right] (X, X),
\end{eqnarray*}
where $R$ denotes the scalar curvature of $(M, g)$. This is precisely the condition stated in item (vi) of Proposition~\ref{properties}.

\end{remark}

Proceeding, it follows from Eqs. \eqref{s1} and (\ref{rrr}) that 
\begin{eqnarray}\label{ab01}
\nabla^{2} f = f\left(Ric + 2E^{\flat}\otimes E^{\flat} - \frac{R}{(n-1)}g\right)
\end{eqnarray}
and 
\begin{eqnarray}\label{ab02}
\Delta f = \left(\frac{(n-2)}{(n-1)}R - 2\Lambda\right)f.
\end{eqnarray}

The next result is a divergence-type formula in the setting of electrostatic systems, primarily inspired by the classical Robinson-Shen identity \cite{sh}; see also \cite{leandro2023, leandro2024} for related formulas.

\begin{lemma}
\label{lemmaAk}
Let $(M^{n},\,g,\,f,\,E)$ be an electrostatic system. Then we have:
		\begin{eqnarray}\label{notnew}
			div \left[\frac{1}{f}\left(\nabla|\nabla f|^2  -4f\langle\nabla f,\,E\rangle E^{\flat} + \frac{2Rf}{n(n-1)}\nabla f\right)\right]&=& 2f|\mathring{Ric}|^2 + 4f\mathring{Ric}(E,\,E) \nonumber\\&&+ \frac{n-2}{n}\langle\nabla R,\nabla f\rangle .
		\end{eqnarray}
	\end{lemma}

	\begin{proof}
	One easily verifies that 
	
	$$div (\mathring{Ric}(\nabla f))=div\mathring{Ric}(\nabla f)+\langle \mathring{Ric},\,\nabla^{2}f\rangle.$$ This jointly with the twice-contracted second Bianchi identity ($2div\,Ric=\nabla R$) and \eqref{ab01} yields

	\begin{eqnarray}
			\label{Identdivric}
			div (\mathring{Ric}(\nabla f))&=& f|\mathring{R}ic|^2+\frac{n-2}{2n}\langle\nabla R,\nabla f\rangle + 2f\mathring{R}ic(E,\,E).
		\end{eqnarray}
	
	On the other hand, it follows from (\ref{ab01}) that 
	
		\begin{equation}
			\label{RSTipeIdentityEq1}
			2\mathring{Ric}(\nabla f)=\frac{1}{f}\nabla|\nabla f|^2  -4\langle\nabla f,\,E\rangle E^{\flat} + \frac{2R}{n(n-1)}\nabla f,
		\end{equation}  Plugging this into (\ref{Identdivric}), one sees that

		\begin{eqnarray*}
			div \left[\frac{1}{f}\left(\nabla|\nabla f|^2  - 4f\langle\nabla f,\,E\rangle E^{\flat} + \frac{2Rf}{n(n-1)}\nabla f\right)\right]\nonumber\\ =2f|\mathring{Ric}|^2+\frac{n-2}{n}\langle\nabla R,\nabla f\rangle + 4f\mathring{Ric}(E,\,E),
		\end{eqnarray*}  as asserted. 
	\end{proof}

As a consequence of the previous result, we obtain the following integral identity.

	\begin{proposition}
	\label{propprop}
		Let $(M^{n},\,g,\,f,\,E)$ be a compact electrostatic system with boundary $\partial M.$  Then we have:
		\begin{eqnarray}\label{divformula}
			\int_{M}\left[\frac{1}{f}\vert\mathring{\nabla}^2f\vert^{2} + f\vert\mathring{Ric}\vert^{2}\right]dV_{g}	&=&	\int_{\partial M}\vert\nabla f\vert R^{\partial M} dA_{g} \nonumber\\
			&&+ \frac{1}{n}\int_{M}\left( 4(n-1)f\vert E\vert^{4}+(n-2)R\Delta f\right) dV_{g}.
		\end{eqnarray}
		In particular, if $\Lambda\geq 0,$ then
		\begin{equation}\label{geoinq}
\begin{aligned}
\frac{4(n-1)}{n}\int_{M}f\Bigg\{ &\vert E\vert^{2} +  \dfrac{(n-2)^2}{4(n - 1)^2}
\left[\sqrt{R\left(R + \dfrac{8(n-1)^2}{(n-2)^3}\Lambda\right)} + R\right]\Bigg\} \\
&\times\left\{\dfrac{(n-2)^2}{4(n - 1)^2}
\left[\sqrt{R\left(R + \dfrac{8(n-1)^2}{(n-2)^3}\Lambda\right)} - R\right] 
- \vert E\vert^{2} \right\} dV_{g} \\
&\leq \int_{\partial M}\vert\nabla f\vert R^{\partial M} \, dA_{g}.
\end{aligned}
\end{equation}
Moreover, equality holds in (\ref{geoinq}) if and only if $(M^{n},\,g,\,f,\,E)$ is isometric to the de Sitter system.
	\end{proposition}

	\begin{proof}
	On integrating (\ref{Identdivric}) and using Stokes' theorem, we obtain

			\begin{eqnarray}
			\label{efg5}
			\int_{M} div\big(\mathring{Ric}(\nabla f)\big)\,dV_{g}&=&\int_{M}\left[f|\mathring{R}ic|^{2}+2f\mathring{R}ic(E,\,E)\right]dV_{g}\nonumber\\
			&&+\frac{n-2}{2n}\int_{M}\langle \nabla R,\,\nabla f\rangle dV_{g}\nonumber\\&=& \int_{M}\left[f|\mathring{R}ic|^{2}+2f\mathring{R}ic(E,\,E)\right]dV_{g}\nonumber\\
			&&+\frac{n-2}{2n}\Big(\int_{\partial M}R\langle \nabla f,\nu\rangle dS_{g}-\int_{M}R\Delta f dV_{g}\Big)\nonumber\\&=& \int_{M}\left[f|\mathring{R}ic|^{2}+2f\mathring{R}ic(E,\,E)\right]dV_{g}\nonumber\\
			&&-\frac{n-2}{2n}\Big(\int_{\partial M}R|\nabla f| dA_{g}+\int_{M}R\Delta f dV_{g}\Big).
		\end{eqnarray}

		On the other hand, it follows from \eqref{gausseq} and Stokes' theorem that
		\begin{eqnarray}
			\label{4eq2}
			\int_{M} div\big(\mathring{R}ic(\nabla f)\big)\,dV_{g}&=&\int_{\partial M}\langle \mathring{R}ic(\nabla f),\nu\rangle dA_{g}\nonumber\\&=&-\int_{\partial M}|\nabla f|\mathring{R}ic(\nu,\,\nu) dA_{g}\nonumber\\&=& -\int_{\partial M}|\nabla f|\left(\frac{n-2}{2n}R -\frac{1}{2}R^{\partial M}\right) dA_{g}.
		\end{eqnarray} This combined with (\ref{efg5}) gives

		\begin{eqnarray}\label{eqmassa}
			\int_{\partial M}\vert\nabla f\vert R^{\partial M} dA_{g}&=& 2\int_{M}\left(f|\mathring{Ric}|^{2}+2f\mathring{R}ic(E,\,E)\right)\,dV_{g}\nonumber\\
			&& -\frac{(n-2)}{n}\int_{M}R\Delta f \,dV_{g}.
		\end{eqnarray}

		Next, since $\mathring{Ric}=Ric-\frac{R}{n}g,$ one deduces from \eqref{ab01} that
		
		$$\nabla^2 f = f\left(\mathring{Ric}+2 E^{\flat}\otimes E^{\flat} -\frac{R}{n(n-1)}g\right),$$ and consequently,

		\begin{eqnarray*}
			\vert\nabla^2f\vert^{2} = f^2\vert\mathring{Ric}\vert^{2} + 4f^2\mathring{Ric}(E,\,E) + 4f^2\vert E\vert^{4} - \frac{4f^2R\vert E\vert^{2}}{n(n-1)} + \frac{f^2R^2}{n(n-1)^2},
		\end{eqnarray*} which can be rewrite as

		\begin{eqnarray*}
			\vert\mathring{\nabla}^2f\vert^{2} + \frac{(\Delta f)^{2}}{n} = f^2\vert\mathring{R}ic\vert^{2} + 4f^2\mathring{R}ic(E,\,E) + 4f^2\vert E\vert^{4} - \frac{4f^2R\vert E\vert^{2}}{n(n-1)} + \frac{f^2R^2}{n(n-1)^2},
		\end{eqnarray*} where we have used that $\mathring{\nabla^2} f = \nabla^2 f-\frac{\Delta f}{n}g.$ By using \eqref{ab02}, we get 
		
		\begin{eqnarray*}
			2\mathring{Ric}(E,\,E) &=& \frac{1}{2f^2}\vert\mathring{\nabla}^2f\vert^{2} - \frac{1}{2}\vert\mathring{Ric}\vert^{2} + \frac{1}{2n}\left(\frac{(n-2)}{(n-1)}R - 2\Lambda\right)^2   - 2\vert E\vert^{4} \nonumber\\
			&&+ \frac{2 R\vert E\vert^{2}}{n(n-1)} - \frac{R^2}{2n(n-1)^2}.
		\end{eqnarray*} Rearranging terms, we obtain

		\begin{eqnarray*}
			\vert\mathring{Ric}\vert^{2}+ 2\mathring{Ric}(E,\,E) &=& \frac{1}{2f^2}\vert\mathring{\nabla}^2f\vert^{2} + \frac{1}{2}\vert\mathring{Ric}\vert^{2} \nonumber\\
			&&+ \frac{(n-3)}{2n(n-1)}R^2 - \frac{2(n-2)}{n(n-1)}\Lambda R + \frac{2\Lambda^2}{n} - 2\vert E\vert^{4} \nonumber\\&&+ \frac{2 R\vert E\vert^{2}}{n(n-1)}.
		\end{eqnarray*} Now, we use \eqref{rrr} to infer

		\begin{eqnarray*}
			\vert\mathring{Ric}\vert^{2}+ 2\mathring{Ric}(E,\,E) &=& \frac{1}{2f^2}\vert\mathring{\nabla}^2f\vert^{2} + \frac{1}{2}\vert\mathring{R}ic\vert^{2} \nonumber\\
			&&+ \frac{(n-3)}{2n(n-1)}(2R\vert E\vert^{2} + 2\Lambda R) - \frac{2(n-2)}{n(n-1)}\Lambda R + \frac{1}{n}(\Lambda R - 2\Lambda\vert E\vert^{2}) \nonumber\\
			&&- 2\vert E\vert^{4} + \frac{2 R\vert E\vert^{2}}{n(n-1)} \nonumber\\
			&=& \frac{1}{2f^2}\vert\mathring{\nabla}^2f\vert^{2} + \frac{1}{2}\vert\mathring{R}ic\vert^{2} - \frac{2(n-1)}{n}\vert E\vert^{4},
		\end{eqnarray*} so that

		\begin{eqnarray*}
			2(f\vert\mathring{R}ic\vert^{2}	+ 2f\mathring{R}ic(E,\,E)) = \frac{1}{f}\vert\mathring{\nabla}^2f\vert^{2} + f\vert\mathring{R}ic\vert^{2} - \frac{4(n-1)}{n}f\vert E\vert^{4}.
		\end{eqnarray*} Besides, upon integrating the above expression, we use (\ref{eqmassa}) to deduce
		\begin{eqnarray*}
			&&	\int_{M}\left[\frac{1}{f}\vert\mathring{\nabla}^2f\vert^{2} + f\vert\mathring{R}ic\vert^{2}\right]dV_{g}\nonumber\\
			&=&
			\int_{\partial M}\vert\nabla f\vert R^{\partial M} dA_{g} + \frac{1}{n}\int_{M}\left[4(n-1)f\vert E\vert^{4}+(n-2)R\Delta f\right]dV_{g},
		\end{eqnarray*} which proves \eqref{divformula}.

		We now address the inequality (\ref{geoinq}). By Definition \ref{def1}, one sees that
		\begin{eqnarray*}
			4(n-1)f\vert E\vert^{4}+(n-2)R\Delta f = 4(n-1)f\vert E\vert^{4}+   \dfrac{2(n-2)R}{n - 1}\left[(n - 2)|E|^2 - 
			\Lambda\right]f.
		\end{eqnarray*} In particular, we consider
		\begin{eqnarray*}
			F(\vert E\vert^{2}) := 4(n-1)\vert E\vert^{4}+   \dfrac{2(n-2)^2R}{(n - 1)}|E|^2 - 
			\dfrac{2(n-2)R}{(n - 1)}\Lambda,
		\end{eqnarray*} where $F$ is a quadratic function of $\vert E\vert^{2}$. Computing its discriminant, one deduces that

		\begin{eqnarray*}
			\triangle = \dfrac{4(n-2)^4R^2}{(n - 1)^2} + 	32(n-2)R\Lambda  =\dfrac{4(n-2)^4R}{(n - 1)^2}\left(R + \frac{8(n-1)^2}{(n-2)^3}\Lambda\right),
		\end{eqnarray*}which is nonnegative because $\Lambda\geq 0.$ Consequently, $F(\vert E\vert^{2}) =0 $ if and only if
		\begin{eqnarray*}
			\vert E\vert^{2} &=& -\dfrac{(n-2)^2R}{4(n - 1)^2} \pm \dfrac{(n-2)^2}{4(n - 1)^2}\sqrt{R\left(R + \frac{8(n-1)^2}{(n-2)^3}\Lambda\right)}\nonumber\\
			&=&\dfrac{(n-2)^2}{4(n - 1)^2}\left[\pm\sqrt{R\left(R + \frac{8(n-1)^2}{(n-2)^3}\Lambda\right)} - R\right]. \nonumber\\
		\end{eqnarray*} From this, it follows that 
		\begin{eqnarray*}
			F(\vert E\vert^{2}) &=& 4(n-1)\left\{\vert E\vert^{2} +  \dfrac{(n-2)^2}{4(n - 1)^2}\left[\sqrt{R\left(R + \frac{8(n-1)^2}{(n-2)^3}\Lambda\right)} +R\right]\right\}\nonumber\\
			&&\times\left\{\vert E\vert^{2} -  \dfrac{(n-2)^2}{4(n - 1)^2}\left[\sqrt{R\left(R + \frac{8(n-1)^2}{(n-2)^3}\Lambda\right)} - R\right]\right\}.
		\end{eqnarray*} Plugging this into \eqref{divformula} yields
	
		\begin{eqnarray}\label{prop2eqb}
			&&	\int_{M}\left[\frac{1}{f}\vert\mathring{\nabla}^2f\vert^{2} + f\vert\mathring{R}ic\vert^{2}\right]dV_{g}\nonumber\\
			&=&
			\int_{\partial M}\vert\nabla f\vert R^{\partial M} dA_{g} \nonumber\\
			&&+ \frac{4(n-1)}{n}\int_{M}f\left\{\vert E\vert^{2} +  \dfrac{(n-2)^2}{4(n - 1)^2}\left[\sqrt{R\left(R + \frac{8(n-1)^2}{(n-2)^3}\Lambda\right)} +R\right]\right\}\nonumber\\
			&&\times\left\{\vert E\vert^{2} -  \dfrac{(n-2)^2}{4(n - 1)^2}\left[\sqrt{R\left(R + \frac{8(n-1)^2}{(n-2)^3}\Lambda\right)} - R\right]\right\}dV_{g},
		\end{eqnarray} so that
		\begin{eqnarray}
		\label{plkmnjh11}
			&& \frac{4(n-1)}{n}\int_{M}f\left\{\vert E\vert^{2} +  \dfrac{(n-2)^2}{4(n - 1)^2}\left[\sqrt{R\left(R + \frac{8(n-1)^2}{(n-2)^3}\Lambda\right)} +R\right]\right\}\nonumber\\
			&&\times\left\{\dfrac{(n-2)^2}{4(n - 1)^2}\left[\sqrt{R\left(R + \frac{8(n-1)^2}{(n-2)^3}\Lambda\right)} - R\right] -\vert E\vert^{2} \right\}dV_{g}\leq 	\int_{\partial M}\vert\nabla f\vert R^{\partial M} dA_{g},
		\end{eqnarray} as asserted.

Finally, equality holds in (\ref{plkmnjh11})  if and only if $\mathring{\nabla}^2 f = 0.$ Hence, the rigidity conclusion follows from Reilly’s generalization of Obata’s theorem to compact manifolds with boundary (see \cite[Lemma 3]{reilly2}).

\end{proof}

The next lemma establishes a key equivalence within our framework by using the condition that the gradient of the potential \(\nabla f\) is parallel to the electric field $E,$ a condition that holds in all known examples (see Section~\ref{examples}). In particular, it will be used to establish a divergence-free tensor, which provides a suitable setting for applying integral identities such as in Theorem~\ref{Poho_gen}.

\begin{remark}
\label{remarkNEW} 
As previously observed, even in the three-dimensional case this condition that the gradient of the potential \(\nabla f\) is parallel to the electric field $E$ appears to be necessary. Indeed, as in the proof of \cite[Proposition~21]{tiarlos}, by setting 
\begin{equation}
\label{tensorTT}
T=Ric-\frac{R}{2}g + 2E^\flat\otimes E^{\flat}-|E|^2 g,
\end{equation} one obtains that  
\begin{eqnarray}
div\,T &=& g^{jk}\nabla_{k}\left(R_{ij}-\frac{R}{2}g_{ij}+2E_{i}E_{j}-|E|^2 g_{ij}\right)\nonumber\\&=& g^{jk}\nabla_{k}R_{ij}-\frac{1}{2}\nabla_{i}R+2g^{jk}(\nabla_{k}E_{i})E_{j}+2g^{jk}(\nabla_{k}E_{j})E_{i}-\nabla_{i}|E|^2\nonumber\\&=& 2(\nabla_{j}E_{i})E_{j}-2(\nabla_{i}E_{j})E_{j}\nonumber\\&=&2(\nabla_{j}E_{i}-\nabla_{i}E_{j})E_{j},
\end{eqnarray} where we have used the twice-contracted second Bianchi identity and the fact that $div(E)=0.$ In particular, $div\,T=0$ if and only if $\nabla_{j}E_{i}=\nabla_{i}E_{j}.$ This fact motivates our condition that the gradient of the potential \(\nabla f\) is parallel to the electric field $E$ in the next lemma. 
\end{remark}

	\begin{lemma}\label{theo0}
		Let $(M^n,\,g,\,f,\,E)$, $n\geq3$, be an electrostatic system. Then $E$ is parallel to $\nabla f$ if and only if $div\, T=0.$
	\end{lemma}

\begin{proof}
By using \eqref{ab01}, in tensorial notation\footnote{The Einstein convention of summing over the repeated indices is adopted.}, we notice that
	\begin{eqnarray*}
		\nabla_i\nabla_i\nabla_jf  &=& 	\nabla_ifR_{ij} + 2	\nabla_ifE_iE_j - \frac{R}{(n-1)}	\nabla_jf + f\nabla_iR_{ij} \\
		&&+ 2fE_i\nabla_iE_j - \frac{f}{(n-1)}	\nabla_jR,
	\end{eqnarray*} where we used that $div(E)=0$. Moreover, by the twice contracted second Bianchi identity, we get
	\begin{eqnarray}
	\label{eqnjk123}
		\nabla_i\nabla_i\nabla_jf  &=& 	\nabla_ifR_{ij} + 2	\nabla_if E_iE_j - \frac{R}{(n-1)}	\nabla_jf + \frac{1}{2}f\nabla_jR\nonumber \\
		&& + 2fE_i\nabla_iE_j - \frac{f}{(n-1)}	\nabla_jR\nonumber\\
		&=& 	\nabla_ifR_{ij} + 2\langle\nabla f,\,E\rangle E_j - \frac{R}{(n-1)}	\nabla_jf + \frac{(n-3)}{2(n-1)}f\nabla_jR\nonumber\\
		&&  + 2fE_i\nabla_i E_j.
	\end{eqnarray}

	On the other hand, the Ricci identity:
	\begin{eqnarray*}
		\nabla_i\nabla_j\nabla_kf - 	\nabla_j\nabla_i\nabla_kf = R_{ijkl}\nabla_lf
	\end{eqnarray*} yields
	\begin{eqnarray*}
		g^{ik}\nabla_i\nabla_j\nabla_kf - 	g^{ik}\nabla_j\nabla_i\nabla_kf = g^{ik}R_{ijkl}\nabla_lf,
	\end{eqnarray*}
	so that
	\begin{eqnarray*}
		\nabla_k\nabla_k\nabla_j f - 	\nabla_j\Delta f = R_{jl}\nabla_lf.
	\end{eqnarray*} Rearranging the indices and using (\ref{eqnjk123}), one obtains that

	\begin{eqnarray}
	\label{eqnjk123111}
		\nabla_j\Delta f 
		= 2\langle\nabla f,\,E\rangle E_j - \frac{R}{(n-1)}	\nabla_jf + \frac{(n-3)}{2(n-1)}f\nabla_jR + 2fE_i\nabla_i E_j.
	\end{eqnarray} Now, substituting \eqref{ab02} into (\ref{eqnjk123111}) gives 
	\begin{eqnarray*}
		\nabla_j\left[\left(\frac{(n-2)}{(n-1)}R - 2\Lambda\right)f\right]
		&=& 2\langle\nabla f,\,E\rangle E_j - \frac{R	\nabla_jf}{(n-1)} + \frac{(n-3)}{2(n-1)}f\nabla_jR\nonumber\\&& + 2fE_i\nabla_i E_j,
	\end{eqnarray*} which can be written as
	\begin{eqnarray*}
		\frac{1}{2}f\nabla_jR+\left(R - 2\Lambda\right)\nabla_jf= 2\langle\nabla f,\,E\rangle E_j + 2fE_i\nabla_i E_j,
	\end{eqnarray*} From this, it follows that
	\begin{eqnarray*}
		\nabla_j\left[(R-2\Lambda)f^{2}\right]= 4f\langle\nabla f,\,E\rangle E_j + 4f^2E_i\nabla_i E_j = 4fdiv(fE^{\flat}\otimes E^{\flat}).
	\end{eqnarray*} Now, we use (\ref{rrr}) to infer 

	\begin{eqnarray*}
		\nabla_j(f^{2}\vert E\vert^{2})=\nabla_i(f^{2}\vert E\vert^{2}g_{ij}) = 2fg^{ik}\nabla_k(fE_i E_j).
	\end{eqnarray*} In particular, we have
	$$\nabla_{i}(f^2 |E|^2 g_{ij})=\nabla_{i}f (f|E|^2 g_{ij}) + f\nabla_{i}(f|E|^2 g_{ij})=fg^{ik}\nabla_k(2fE_i E_j).$$ Since $f>0$ in the interior of $M^n,$ one deduces that  
	\begin{eqnarray*}
		\vert E\vert^{2}\nabla_jf= \nabla_i(2fE_i E_j - f\vert E\vert^{2}g_{ij}).
	\end{eqnarray*} Therefore,
	\begin{eqnarray*}
		\vert E\vert^{2}\nabla_jf= (2E_i E_j - \vert E\vert^{2}g_{ij})\nabla_i f + f\nabla_i(2E_i E_j - \vert E\vert^{2}g_{ij}),
	\end{eqnarray*} which can be reformulated as		
	\begin{eqnarray}\label{fom}
		0= 2(E_i E_j - \vert E\vert^{2}g_{ij})\nabla_i f + f\nabla_i(2E_i E_j - \vert E\vert^{2}g_{ij}).
	\end{eqnarray}

If $E$ is parallel to $\nabla f,$ there exists a function $h$ on $M^n$ such that $E = h\nabla f.$ Thus, one easily verifies that
	\begin{eqnarray*}
		(E_i E_j - \vert E\vert^{2}g_{ij})\nabla_i f &=& \langle E,\,\nabla f\rangle E_j - \vert E\vert^{2}\nabla_j f\nonumber\\&=& h^2\langle \nabla f,\,\nabla f\rangle \nabla_jf - h^2\vert \nabla f\vert^{2}\nabla_j f = 0.
	\end{eqnarray*}	
Then, it follows from (\ref{fom}) that 
	\begin{eqnarray}
	\label{kjhgf019}
		\nabla_i(2E_i E_j - \vert E\vert^{2}g_{ij})=div \left(2E^\flat\otimes E^{\flat} - \vert E\vert^{2}g\right)=div\,T=0.
	\end{eqnarray}

Conversely, if $\nabla_i(2E_i E_j - \vert E\vert^{2}g_{ij})=0,$ one obtains from (\ref{fom}) that
	\begin{eqnarray*}
		0= 2(E_i E_j - \vert E\vert^{2}g_{ij})\nabla_i f,
	\end{eqnarray*}
	and hence,
	\begin{eqnarray*}
		0= 2(E_i E_j - \vert E\vert^{2}g_{ij})\nabla_i f\nabla_j f,
	\end{eqnarray*}
	which is equivalent to 
	\begin{eqnarray*}
		\vert E\vert^{2}\vert\nabla f\vert^{2} = \langle E,\,\nabla f\rangle^{2}.
	\end{eqnarray*} Then, $E$ is parallel to $\nabla f$. So, the proof is completed. 
\end{proof}

	\begin{remark}
	An advantage of the proof of Lemma~\ref{theo0} is that it does not rely on the condition $d(fE^{\flat})=0.$ In other words, the result holds in a more general setting.
	\end{remark}

Next, we establish the following proposition, which holds in a  general context, allowing \( M^n \) to be non-compact or without boundary.

\begin{proposition}
\label{propKa}
   	Let $(M^n,\, g,\, f,\, E)$ be an electrostatic system such that $E$ is parallel to $\nabla f.$ If $\nabla^{2}f=\dfrac{\Delta f}{n}g,$ then $|E|$ and the scalar curvature $R$ are constants. 
	
	\end{proposition}

\begin{proof}
Initially, by using $\nabla^{2}f=\frac{\Delta f}{n}g$ into the first equation of Definition \ref{def1}, we obtain

	\begin{eqnarray}\label{def1Hess}
			\frac{\Delta f}{n}g = f\left({Ric} - \dfrac{2}{n-1}\Lambda g + 2E^\flat\otimes E^\flat - \dfrac{2}{n-1}|E|^2 g\right),
		\end{eqnarray} so that
\begin{eqnarray*}
			\dfrac{2}{n(n - 1)}\left[(n - 2)|E|^2 - 
			\Lambda\right]g = ({Ric} - \dfrac{2}{n-1}\Lambda g + 2E^\flat\otimes E^\flat - \dfrac{2}{n-1}|E|^2g),
		\end{eqnarray*} which can be rewrite as
		
		\begin{eqnarray}\label{6}
			{Ric} - \frac{2\Lambda}{n} g = 2\left(\frac{2}{n}\vert E\vert^{2}g - E^\flat\otimes E^\flat\right).
		\end{eqnarray} So, one easily verifies from (\ref{rrr}) that
		
		\begin{equation}
		\label{lkm671}
		Ric - \frac{R}{n}g = \frac{2}{n}|E|^2 g - 2E^\flat\otimes E^\flat.
		\end{equation} In particular, $(M,\,g)$ is Einstein if and only if 
		\begin{eqnarray*}
			E^\flat\otimes E^\flat = \frac{\vert E\vert^{2}}{n}g.
		\end{eqnarray*}

		Next, by taking the divergence of (\ref{lkm671}) and using Lemma \ref{theo0}, we arrive at
		
		\begin{eqnarray*}
		div\left(Ric - \frac{R}{n}g\right)=2\left(\frac{2-n}{2n}\right)\nabla |E|^2.
		\end{eqnarray*} By applying the twice-contracted second Bianchi identity together with equation (\ref{rrr}), we obtain $$\frac{2(n-2)}{n}\nabla |E|^2 =0,$$ which implies that $|E|^2$ is constant. Consequently, the scalar curvature must also be constant.

\end{proof}

To conclude this section, we recall two important integral identities that will be useful in our context. The first one is a very important generalization of the classical Reilly formula \cite[Theorem~1]{reilly2}, recently established by Qiu and Xia \cite{qiu} (see also \cite{LiXia}).

	\begin{proposition}[Generalized Reilly Identity \cite{qiu}]\label{prop1}
		Let $(M^n,\,g)$ be a compact Riemannian manifold with boundary $\partial M.$ Given two functions $f$ and $u$ on $M^n$ and a constant $k,$ we have
		\begin{eqnarray*}
			&&\int_{M}f\left[\left(\Delta u +knu\right)^2 - \vert \nabla^2u+kug\vert^{2}\right]dV_{g} = (n-1)k\int_{M}\left(\Delta f +nkf\right)u^2dV_{g}\\
			&&+\int_{M}\left(\nabla^2f-(\Delta f)g - 2(n-1)kfg +fRic\right)(\nabla u,\,\nabla u)dV_{g}\\
			&&+\int_{\partial M}f\left[2\Delta_{\partial M}u\frac{\partial u}{\partial\nu} + H\left(\frac{\partial u}{\partial\nu}\right)^2 + A(\nabla_{\partial M}u,\,\nabla_{\partial M}u) + 2(n-1)ku\frac{\partial u}{\partial\nu}\right]dA_{g}\\
			&&+\int_{\partial M}\frac{\partial f}{\partial\nu}\left[\vert\nabla_{\partial M}u\vert^2 -(n-1)ku^2\right]dA_{g},
		\end{eqnarray*}
		where $H$ and $A$ stand for the mean curvature and second fundamental form of $\partial M$, respectively.
	\end{proposition}

The second is a Pohozaev-type identity, originally observed by Schoen in \cite{schoen} in the context of the Einstein tensor, and subsequently extended to its present form for \( (0,2) \)-tensors $B$ that are locally conserved, that is, \( div\,B = 0 \) (divergence-free).

\begin{theorem}[Generalized Pohozaev-Schoen Identity \cite{gover}]
\label{Poho_gen}
Let \((M^n, \,g)\) be a compact Riemannian manifold and let \(X \in \mathfrak{X}(M)\) be a smooth vector field. If \(B\) is a divergence-free symmetric \( (0,2) \)-tensor, then
\begin{equation}\label{Poho_eq}
\int_{M} X(\operatorname{tr} B) \, dV_{g} = \frac{n}{2} \int_{M} \langle \mathring{B}, \mathcal{L}_{X}g \rangle \, dV_{g} - n \int_{\partial M} \mathring{B}(X, \nu) \, dA_{g},
\end{equation}
where \(\operatorname{tr} B\) denotes the trace of \(B\) with respect to \(g\), \(\mathring{B} = B - \frac{\operatorname{tr} B}{n}g\) is the traceless part of \(B\), \(\mathcal{L}_{X}g\) is the Lie derivative of \(g\) in the direction of \(X\), and \(\nu\) is the outward unit normal vector field along \(\partial M\).
\end{theorem}

	\subsection{Examples}
	\label{examples} We now present examples of electrostatic manifolds with boundary for arbitrary dimensions $n\geq 3.$

	\begin{example}[\textbf{Vacuum static systems}]
	\label{Exa1}
The vacuum static systems (those for which \( E \equiv 0 \)) are trivial examples of electrostatic systems. Among them, the {\bf de Sitter system} corresponds to the standard hemisphere \( (\mathbb{S}_{+}^{n}, g) \) endowed with the metric \( g = dr^{2} + \sin^{2}(r) g_{\mathbb{S}^{n-1}} \) and the lapse function by \( f(r) = \cos(r) \) with \( r \leq \pi/2 \) representing the height function. 

Another distinguished vacuum static example is the {\bf Nariai system}, given by \( M^{n} = [0, \pi] \times \mathbb{S}^{n-1} \), with metric \( g = ds^{2} + (n-2) g_{\mathbb{S}^{n-1}} \) and lapse function \( f(s) = \sin(s) \), which yields a compact, oriented electrostatic manifold whose boundary consists of two disconnected components.
\end{example}

\begin{example}[\textbf{Reissner-Nordstrom de Sitter (RNdS)}]
The \((n+1)\)-dimensional RNdS spacetime is a three-parameter family (characterized by the ADM mass \(\mathfrak{M}_{ADM}\), the charge \(Q\), and the cosmological constant \(\Lambda\)) of static, electrically charged solutions to the Einstein equations. Consider the system \((M^n,\,g_{_{\mathrm{RNdS}}},\,f,\,E)\), where
\[
M^n = I \times \mathbb{S}^{n-1},
\]
for some interval \( I \subset \mathbb{R} \) determined by the roots of the static potential
\[
f(r)^2 = 1 - \frac{2\mathfrak{M}_{ADM}}{r^{n-2}} + \frac{Q^2}{r^{2(n-2)}} - \frac{2\Lambda r^2}{n(n-1)},
\]
and
\[
E = \sqrt{\frac{(n-1)(n-2)}{2}} \frac{Q}{r^{n-1}} f(r) \, \partial_r.
\]
The metric is given by
\[
g_{_{\mathrm{RNdS}}} = f(r)^{-2} \, dr^2 + r^2 g_{\mathbb{S}^{n-1}},
\]
where \( g_{\mathbb{S}^{n-1}} \) denotes the standard metric on the unit sphere \( \mathbb{S}^{n-1} \) and \( r \) represents the radial coordinate.  When the cosmological constant vanishes identically, the space is referred to as the {\it Reissner--Nordstr\"om} (RN) space.
\end{example}	
	
\begin{example}[\textbf{Majumdar–Papapetrou (MP)}]	
An important electrostatic system with $\Lambda=0$ is the {\it Majumdar–Papapetrou} (see \cite{chrusciel1999,Lucietti}), which is related to an extremal RN solution. The Majumdar-Papapetrou (MP) solution to the Einstein–Maxwell theory represents the static equilibrium of an arbitrary number of charged black holes whose mutual electric repulsion exactly balances their gravitational attraction. The metric tensor of MP is given by
	\begin{eqnarray}\label{MPmetric}
		\hat{g}=-f^{2}dt^{2}+f^{-2/(n-2)}(dx_1^{2}+\ldots+dx_n^{2}),
	\end{eqnarray}
	in Cartesian coordinates $x=(x_1,\,\ldots,\,x_n)$ and $\widehat{M}^{n+1}=(\mathbb{R}^{n}\backslash\{{a_i}\}_{i=1}^I)\times\mathbb{R},$ for a finite set of points ${a_i}\in\mathbb{R}^{n},$ where
	\begin{eqnarray}\label{harmonic function MP}
		\frac{1}{f({x})}=1+\displaystyle\sum_{i=1}^{I}\frac{m_i}{r_i^{n-2}};\quad r_i=|{x}-{a_i}|,
	\end{eqnarray}
	for some positive constants $m_i$.	
\end{example}

\begin{example}[\textbf{Cold Black Hole}]
The {\it Cold Black Hole system} is defined on \( [0, +\infty) \times \mathbb{S}^{n-1} \) endowed with the metric $g = ds^{2} + \rho^{2} g_{\mathbb{S}^{n-1}}$ and lapse function $f(s) = \sinh(\tau s)$, where 
\[
\tau = \sqrt{\left( \frac{(n-2)^2 Q^{2}}{\rho^{2(n-1)}} - \frac{2\Lambda}{n-1} \right)}.
\] The electric field is given by
\[
E = \sqrt{\frac{(n-1)(n-2)}{2}} \frac{Q}{\rho^{n-1}} \, \partial_s.
\]
\end{example}		

\begin{example}[\textbf{Ultracold Black Hole}]
The {\it ultracold black hole system} is defined on $\left[0,\,+\infty\right)\times\mathbb{S}^{n-1}$ with metric $g = ds^{2} + \rho^{2}g_{\mathbb{S}^{n-1}}$ and lapse function $f(s) = s$. The electric field is given by 
$$E =\sqrt{\dfrac{\Lambda}{(n-2)}}\partial_s,$$
where 
$$Q^2=\frac{1}{n-1}\left(\frac{(n-2)^2}{2\Lambda}\right)^{n-2}=\frac{\rho^{2(n-2)}}{n-1}.$$
\end{example}

\begin{example}[\textbf{Charged Nariai}]
Based on the Nariai system, one can construct the {\it Charged Nariai system} defined on \( \left[0, \frac{\pi}{\gamma} \right] \times \mathbb{S}^{n-1} \) with metric \( g = ds^{2} + \rho^{2} g_{\mathbb{S}^{n-1}} \) and lapse function \( f(s) = \sin(\gamma s) \), where \( \gamma = \sqrt{ \frac{2\Lambda}{n-1} - \frac{(n-2)^2 Q^{2}}{\rho^{2(n-1)}} } \). Moreover, the electric field is given by $$E = \sqrt{\frac{(n-1)(n-2)}{2}} \frac{Q}{\rho^{n-1}} \partial_s.$$ 
\end{example}

\section{Boundary area estimates and rigidity results}

In this section, we establish results concerning boundary area estimates and their implications. We begin with a slightly more general result, from which Theorem~\ref{main}  follows as a particular case, assuming additionally that the boundary is connected.

\begin{theorem}\label{teoAgen}
Let $(M^n,\, g,\, f,\,E)$ be a compact electrostatic system with  boundary $\partial M=\displaystyle\cup_{i=1}^{l} \Sigma_{i}$, where $\Sigma_{i}$ are the connected components of $\partial M$. If $|E|^2\leq\frac{\Lambda}{(n-2)},$ then
\[
\sqrt{\frac{\alpha(n+2)}{2n\beta^{2}}} \left(\frac{\left( \sum_{i} \kappa_{i} |\Sigma_{i}| \right)^3}{\sum_{i} \kappa_{i}^3 |\Sigma_{i}|} \right)^{\frac{1}{2}}\leq Vol(M),
\]
where $ \alpha = \min_{M}\dfrac{2}{n - 1}\left[\Lambda - (n - 2)|E|^2\right]$, $ \beta = \max_{M} \dfrac{2}{n - 1}\left[\Lambda - (n - 2)|E|^2\right]$ and $\kappa_{i}:=|\nabla f|\Big|_{\Sigma_{i}}$ are the surface gravities. Moreover, equality holds if and only if $(M^n,\, g,\, f)$ is isometric to the de Sitter system. 
\end{theorem}

\begin{proof}
By taking $f$ as the potential of the electrostatic system in Proposition \ref{prop1}, and noting that $f = 0$ on $\partial M,$ we deduce
		\begin{eqnarray*}
			&&\int_{M}f\left[\left(\Delta u +knu\right)^2 - \vert \nabla^2u+kug\vert^{2}\right]dV_g = (n-1)k\int_{M}\left(\Delta f +nkf\right)u^2dV_g\\
			&&+\int_{M}\left(\nabla^2f-(\Delta f)g - 2(n-1)kfg +fRic\right)(\nabla u,\,\nabla u)dV_g\\
			&&+\int_{\partial M}\frac{\partial f}{\partial\nu}\left[\vert\nabla_{\partial M}u\vert^2 -(n-1)ku^2\right]dA_g.
		\end{eqnarray*} One sees from Definition \ref{def1} that
		\begin{eqnarray*}
			\nabla^2f(\nabla u,\,\nabla u) = f\left(Ric - \dfrac{2}{n-1}\Lambda g + 2E^\flat\otimes E^\flat - \dfrac{2}{n-1}|E|^2g\right)(\nabla u,\,\nabla u),
		\end{eqnarray*}
		so that
		\begin{eqnarray*}
			&&\int_{M}f\left[\left(\Delta u +knu\right)^2 - \vert \nabla^2u+kug\vert^{2}\right]dV_g = (n-1)k\int_{M}\left(\Delta f +nkf\right)u^2dV_g\\
			&&+\int_{M}\left(2f{Ric} - 2\left(\dfrac{\Lambda}{n-1} +(n-1)k\right)fg + 2f E^\flat\otimes E^\flat - \dfrac{2}{n-1}f|E|^2g-(\Delta f)g	\right)(\nabla u,\,\nabla u)dV_g\\
			&&+\int_{\partial M}\frac{\partial f}{\partial\nu}\left[\vert\nabla_{\partial M}u\vert^2 -(n-1)ku^2\right]dA_g.
		\end{eqnarray*} Taking into account that
		\begin{eqnarray*}
			\Delta f = \dfrac{2}{n - 1}\left((n - 2)|E|^2 - 
			\Lambda\right)f,
		\end{eqnarray*}
		one obtains that
		\begin{eqnarray*}
			&&\int_{M}f\left[\left(\Delta u +knu\right)^2 - \vert \nabla^2u+kug\vert^{2}\right]dV_g = (n-1)k\int_{M}\left(\Delta f +nkf\right)u^2dV_g\\
			&&+\int_{M}2f\left[{Ric}(\nabla u,\,\nabla u) - (n-1)k|\nabla u|^{2}\right]dV_g + \int_{M}2f\left(\langle E,\,\nabla u\rangle^{2} - |E|^2|\nabla u|^{2}	\right)dV_g\\
			&&+\int_{\partial M}\frac{\partial f}{\partial\nu}\left[\vert\nabla_{\partial M}u\vert^2 -(n-1)ku^2\right]dA_g.
		\end{eqnarray*} Now, we take $u=f$ and use the fact that $f=0$ on $\partial M$ to infer
		\begin{eqnarray}\label{1}
			&&\int_{M}f\left[\left(\Delta f  + knf\right)^2 - \vert \nabla^2f+kfg\vert^{2}\right]dV_{g} = (n-1)k\int_{M}\left(\Delta f +nkf\right)f^2dV_{g}\nonumber\\
			&&+\int_{M}2f\left[{Ric}(\nabla f,\,\nabla f) - (n-1)k|\nabla f|^{2}\right]dV_{g} + \int_{M}2f\left(\langle E,\,\nabla f\rangle^{2} - |E|^2|\nabla f|^{2}	\right)dV_{g}.
		\end{eqnarray}
We then use the classical Bochner’s formula:
		\begin{eqnarray*}
			\Delta|\nabla f|^2 = 2Ric(\nabla f,\,\nabla f) + 2|\nabla^{2}f|^{2} + 2\langle\nabla\Delta f,\,\nabla f\rangle
		\end{eqnarray*}
		to compute the second term on the right-hand side of \eqref{1}, observing that
		\begin{eqnarray*}
			&&\int_{M}2f\left[{Ric}(\nabla f,\,\nabla f) - (n-1)k|\nabla f|^{2}\right]dV_{g}\nonumber\\
			&& = \int_{M}f\Bigg[\Delta|\nabla f|^2 - 2|\nabla^{2}f|^{2} - 2\langle\nabla\Delta f,\,\nabla f\rangle-  2(n-1)k|\nabla f|^{2}\Bigg]dV_{g}\nonumber\\
			&& = \int_{M}f\Bigg[\Delta|\nabla f|^2 - 2|\nabla^{2}f|^{2} -  2(n-1)k|\nabla f|^{2}\Bigg]dV_{g} + 2\int_{M}\Delta f(f\Delta f + \vert\nabla f\vert^{2}) dV_{g},
		\end{eqnarray*}
		where we have used that $\Delta f=0$ at $\partial M$. Also by Green's identity and again the fact that $f=0$ in $\partial M,$ one deduces that
		\begin{eqnarray*}
			\int_{M}\left(f\Delta|\nabla f|^{2} - |\nabla f|^{2}\Delta f\right)dV_{g} = \int_{\partial M}\left(f\frac{\partial|\nabla f|^2}{\partial\nu} - |\nabla f|^{2}\frac{\partial f}{\partial\nu}\right)dA_{g} = \sum_{i}\kappa_{i}^3\vert\Sigma_{i}\vert,
\end{eqnarray*}
where $\partial M=\displaystyle\cup_{i=1}^{l} \Sigma_{i}$, $\Sigma_{i}$ are the connected components of $\partial M$ and $\kappa_{i}=|\nabla f|\Big|_{\Sigma_{i}}$  (see Proposition \ref{properties}). Consequently,
		\begin{eqnarray*}
			\int_{M}f\Delta|\nabla f|^{2}dV_{g}  = \sum_{i}\kappa_{i}^3\vert\Sigma_{i}\vert + \int_{M}|\nabla f|^{2}\Delta f dV_{g}.
		\end{eqnarray*}
		Therefore, 
		\begin{eqnarray*}
			&&\int_{M}2f\left[{Ric}(\nabla f,\,\nabla f) - (n-1)k|\nabla f|^{2}\right]dV_{g}\nonumber\\
			&& =\sum_{i}\kappa_{i}^3\vert\Sigma_{i}\vert - \int_{M}f\Bigg[ 2|\nabla^{2}f|^{2} +  2(n-1)k|\nabla f|^{2}\Bigg]dv + 2\int_{M}f(\Delta f)^2 dV_{g}\nonumber\\&& + 3\int_{M}|\nabla f|^{2}\Delta fdV_{g}.
		\end{eqnarray*}
		
		Now, substituting the above identity into \eqref{1}, we obtain
		\begin{eqnarray*}
			&&\int_{M}f\left[\left(\Delta f  + knf\right)^2 - \vert \nabla^2f+kfg\vert^{2}\right]dV_g = (n-1)k\int_{M}\left(\Delta f +nkf\right)f^2dV_g\nonumber\\
			&&+\sum_{i}\kappa_{i}^3\vert\Sigma_{i}\vert - \int_{M}f\Bigg[ 2|\nabla^{2}f|^{2} +  2(n-1)k|\nabla f|^{2}\Bigg]dV_g + 2\int_{M}f(\Delta f)^2 dV_{g}\nonumber\\
			&& + 3\int_{M}|\nabla f|^{2}\Delta fdV_g + \int_{M}2f\left(\langle E,\,\nabla f\rangle^{2} - |E|^2|\nabla f|^{2}	\right)dV_g,
		\end{eqnarray*}
		so that
		\begin{eqnarray*}
			&&\int_{M}2f\left(|E|^2|\nabla f|^{2}-\langle E,\,\nabla f\rangle^{2} 	\right)dV_{g}+\int_{M}f\vert\nabla^2f\vert^{2}dV_{g} \nonumber\\
			&& =\sum_{i}\kappa_{i}^3\vert\Sigma_{i}\vert   + 3\int_{M}|\nabla f|^{2}\Delta fdV_{g}  + \int_{M}f\Delta f(\Delta f - k(n-1)f)dV_{g}\nonumber\\&& - 2(n-1)k\int_{M}f|\nabla f|^{2}dV_{g}.
		\end{eqnarray*} Taking into account that $\vert\nabla^2f\vert^{2} = \vert\mathring{\nabla}^2f\vert^{2} + \frac{(\Delta f)^{2}}{n},$ we infer
	
		\begin{eqnarray}
		\label{plmn4561}
			0 &\leq&\int_{M}2f\left(|E|^2|\nabla f|^{2}-\langle E,\,\nabla f\rangle^{2} 	\right)dV_g+\int_{M}f\vert\mathring{\nabla}^2f\vert^{2}dV_g \nonumber\\
			&=&\sum_{i}\kappa_{i}^3\vert\Sigma_{i}\vert   + 3\int_{M}|\nabla f|^{2}\Delta fdV_g  + (n-1)\int_{M}f\Delta f\left(\frac{\Delta f}{n} - kf\right)dV_g\nonumber\\&& - 2(n-1)k\int_{M}f|\nabla f|^{2}dV_g.
		\end{eqnarray}
		
		We now consider 
		\[
 \dfrac{2}{n - 1} \left[ \Lambda - (n - 2) |E|^2 \right] \geq  \dfrac{2}{n - 1}\min_{M}\left[ \Lambda - (n - 2)  |E|^2 \right] = \alpha,
\] and by (\ref{s1}), one sees that $\Delta f\leq -\alpha f,$ which implies that

\begin{eqnarray}
\label{kljhjçp}
			\int_{M}f(\Delta f)^{2}dV_g \leq -\alpha\int_{M}f^{2}\Delta fdV_g = 2\alpha \int_{M}f\vert\nabla f\vert^{2}dV_g,
		\end{eqnarray}
		where in the last identity we used integration by parts and the fact that $f=0$ on $\partial M.$ 
		
		Substituting (\ref{kljhjçp}) into (\ref{plmn4561}) yields

	\begin{eqnarray}
	\label{4a}
	0&\leq &\int_{M}2f\left(|E|^2|\nabla f|^{2}-\langle E,\,\nabla f\rangle^{2} 	\right)dV_g+\int_{M}f\vert\mathring{\nabla}^2f\vert^{2}dV_g \nonumber\\
&\leq&\sum_{i}\kappa_{i}^3\vert\Sigma_{i}\vert    - 3\alpha\int_{M}f|\nabla f|^{2}dV_g  + \frac{2\alpha(n-1)}{n}\int_{M}f\vert\nabla f\vert^{2}dV_g \nonumber\\ 
			&&- (n-1)k\int_Mf^{2}\Delta f dV_g - 2(n-1)k\int_{M}f|\nabla f|^{2}dV_g\nonumber\\
			&=&\sum_{i}\kappa_{i}^3\vert\Sigma_{i}\vert     - \frac{\alpha(n+2)}{n}\int_{M}f\vert\nabla f\vert^{2}dV_g .
		\end{eqnarray} 
		
		Now, we need to estimate the term $\int_{M}f|\nabla f|^{2}dV_{g}.$ To do so, we first observe that our assumption jointly with (\ref{rrr}) imply  $- \Delta f\geq0$ and hence, by by H\"older’s inequality and integration by parts, one sees that
		\begin{eqnarray}\label{2}
	\left(\int_{M}f\Delta f dV_{g}\right)^{2} &\leq & \int_{M}(-\Delta f)dV_{g} \int_{M}f^{2}(-\Delta f)dV_{g}  \nonumber\\&=& \sum_{i}\kappa_{i}\vert\Sigma_{i}\vert \int_{M}f^{2}(-\Delta f)dV_{g} \nonumber\\
			&=& 2\sum_{i}\kappa_{i}\vert\Sigma_{i}\vert \int_{M}f|\nabla f|^{2}dV_{g}.
					\end{eqnarray} To proceed, we set 
					
						\[
	\beta = \dfrac{2}{n - 1}\max_{M}\left[ \Lambda - (n - 2)  |E|^2 \right] \geq	\dfrac{2}{n - 1} \left[ \Lambda - (n - 2) |E|^2 \right],
		\] so that $\Delta f \geq -\beta f.$ Again, by H\"older’s inequality, one obtains that
		
		\begin{eqnarray}
		\label{holder1}
			\left(\int_{M}\Delta f\,dV_{g}\right)^{2} \leq Vol(M)\int_{M}(\Delta f)^2\,dV_{g} \leq -\beta Vol(M)\int_{M} f\Delta f\,dV_{g}.
		\end{eqnarray} From this, it follows that
		\begin{eqnarray}\label{3}
			\left(\sum_{i}\kappa_{i}\vert\Sigma_{i}\vert\right)^{4} =  \left(\int_{M}\Delta f dV_{g}\right)^{4} \leq \beta^2 Vol(M)^2\left(\int_{M} f\Delta f dV_{g}\right)^{2}.
		\end{eqnarray} This jointly with  \eqref{2} yields
		\begin{eqnarray*}
			\frac{\left(\sum_{i}\kappa_{i}\vert\Sigma_{i}\vert\right)^{3}}{2\beta^{2} Vol(M)^2} \leq \int_{M}f\vert\nabla f\vert^{2}dV_{g}.
		\end{eqnarray*} Now, it suffices to use  \eqref{4a} to infer
		\begin{eqnarray*}
\frac{\alpha(n+2)}{n}\int_{M}f\vert\nabla f\vert^{2}dV_g \leq	\sum_{i}\kappa_{i}^3\vert\Sigma_{i}\vert, 
		\end{eqnarray*}
consequently, 
	\begin{eqnarray*}
	\frac{\alpha(n+2)}{2n\beta^{2} Vol(M)^2}\left(\sum_{i}\kappa_{i}\vert\Sigma_{i}\vert\right)^{3} \leq	\sum_{i}\kappa_{i}^3\vert\Sigma_{i}\vert,
	\end{eqnarray*} which proves the asserted inequality.

We now consider the case of equality. In this situation, it follows from \eqref{4a} that $E$ is parallel to $\nabla f$, and $\mathring{\nabla}^{2}f=0.$ Hence, we use Proposition \ref{propKa} to deduce that $M^n$ has constant scalar curvature. Consequently, since the boundary  \(\partial M\) is totally geodesic, we may then invoke Reilly's generalization of Obata’s theorem (see \cite[Lemma 3]{reilly2}) to conclude that $(M^n,\,g,\,f)$ is isometric to the de Sitter system. This finishes the proof of the theorem. 
\end{proof}

\begin{proof}[{\bf Proof of Theorem \ref{main}}]
It suffices to simplify the (unique) constant $\kappa$ appearing in Theorem~\ref{teoAgen}.
\end{proof}

Proceeding, we present the proof of Theorem \ref{Crusc_th}.

\begin{proof}[{\bf Proof of Theorem \ref{Crusc_th}}]
		We first consider the tensor $T$ given by 
		\[
		T := {Ric} - \frac{R}{2} g + 2 E^{\flat} \otimes E^{\flat} - |E|^{2} g.
		\]  It follows from Lemma \ref{theo0} that $T$ satisfies $div\,T=0.$ Moreover, taking the trace of $T,$ we have
		
		\begin{eqnarray*}
			tr\, T &=&\frac{2-n}{2}R+(2-n)|E|^{2}\\
			&=&
			\frac{2-n}{2}\left(2\Lambda+2|E|^{2}\right)+(2-n)|E|^{2}\\
			&=&-(n-2)\Lambda+2(2-n)|E|^{2},
		\end{eqnarray*} where we have used (\ref{rrr}). Besides, one sees that 
				\begin{eqnarray*}
			\mathring{T}&=&T-\frac{tr\,T}{n}g\\
			&=&Ric-\frac{R}{n}g+2 E^{\flat} \otimes E^{\flat} - \frac{2}{n}|E|^{2} g.
		\end{eqnarray*} By choosing $X=\nabla f$ in \eqref{Poho_eq}, using integration by parts and the fact that $\nu=-\frac{\nabla f}{|\nabla f|}$ along $\partial M,$ one obtains that
		
		\begin{eqnarray}
		\label{eqnbv86000}
				\int_{M}\nabla f (tr\, T) dV_g &=&\int_{M}\left\langle\nabla f,\nabla\left(-(n-2)\Lambda+(4-2n)|E|^{2}\right)\right\rangle dV_g\nonumber\\
				&=&2(n-2)\int_{M}|E|^{2}\Delta f\, dV_g-2(n-2)\int_{\partial M}|E|^{2}\langle\nabla f,\, \nu\rangle\,dA_g\nonumber\\
				&=&2(n-2)\int_{M}f|E|^{2}\left(2\left(\frac{n-2}{n-1}\right)|E|^{2}-\frac{2}{n-1}\Lambda\right)\,dV_g\nonumber\\
				&& + 2(n-2)\int_{\partial M}|E|^{2}|\nabla f|dA_g.
			\end{eqnarray}
Next, since $\mathcal{L}_{\nabla f}g=2\nabla^{2}f$, we use the first equation in Definition \ref{def1} to deduce

\begin{equation}
\label{eqnbv860}
\int_{M}\langle\mathring{T},\mathcal{L}_{\nabla f}g\rangle dV_g=2\int_{M}f|\mathring{T}|^{2}dV_g.
\end{equation}
At the same time, by the Gauss equation and Proposition \ref{properties}, one sees that

			\begin{eqnarray}
			\label{eqnbv8602}
				\int_{\partial M}\mathring{T}(\nabla f,\,\nu)dA_g&=&\int_{\partial M}\left[Ric(\nabla f,\, \nu)-\frac{R}{n}\langle\nabla f,\, \nu\rangle+2\langle E,\nabla f\rangle\langle E,\nu\rangle-\frac{2}{n}|E|^{2}\langle\nabla f,\, \nu\rangle\right]\,dA_g\nonumber\\
				&=&\int_{\partial M}\left[-|\nabla f|\left(Ric(\nu,\,\nu)-\frac{R}{n}\right)-\frac{2}{|\nabla f|}\langle E,\,\nabla f\rangle^{2}+\frac{2}{n}|E|^{2}|\nabla f|\right]\,dA_g\nonumber\\
				&=&\int_{\partial M}\left[-|\nabla f|\left(-\frac{R^{\partial M}}{2}+\frac{n-2}{2n}R\right)+\frac{2-2n}{n}|E|^{2}|\nabla f|\right]\,dA_g.
			\end{eqnarray}

			Substituting (\ref{eqnbv86000}), (\ref{eqnbv860}) and (\ref{eqnbv8602}) into \eqref{Poho_eq} of Theorem \ref{Poho_gen}, we get

		\begin{eqnarray*}
			& &2(n-2)\int_{M}f|E|^{2}\left[2\left(\frac{n-2}{n-1}\right)|E|^{2}-\frac{2}{n-1}\Lambda\right]dV_g+2(n-2)\int_{\partial M}|E|^{2}|\nabla f|dA_g\\
			&=&n\int_{M}f|\mathring{T}|^{2}dV_g+n\int_{\partial M}|\nabla f|\left[-\frac{R^{\partial M}}{2}+\frac{n-2}{2n}R+\frac{2(n-1)}{n}|E|^{2}\right]dA_g,
		\end{eqnarray*} which implies that
		
		\begin{eqnarray}
		&&\int_{M} \left[ f |\overset{\circ}{T}|^{2} + \frac{2(n-2)}{n} |E|^{2} \left( \frac{2}{n-1}\Lambda - \frac{2(n-2)}{n-1} |E|^{2} \right) f \right]dV_g\nonumber\\ &\leq & \int_{\partial M}|\nabla f|\left[ \frac{R^{\partial M}}{2}-\frac{(n-2)}{2n}R-\frac{2}{n}|E|^2\right]\,dA_{g}.
\end{eqnarray} Hence, taking into account that \( \kappa_{i} = |\nabla f| \big|_{\Sigma_{i}} \), where \( \Sigma_{i} \) are the connected components of \( \partial M \), we use the assumption and Eq. \eqref{charge} to infer
		
		\begin{eqnarray}\label{alan}
			0 &\leq& \int_{M} \left[ f |\overset{\circ}{T}|^{2} + \frac{2(n-2)}{n} |E|^{2} \left( \frac{2}{n-1}\Lambda - \frac{2(n-2)}{n-1} |E|^{2} \right) f \right]dV_g \nonumber\\
			&=& \sum_{i=1}^{l} \kappa_{i} \int_{\Sigma_{i}} \left[ \frac{R^{\Sigma_{i}}}{2} - \frac{n-2}{2n} R - \frac{2}{n} |E|^{2} \right]dA_g\\
			&=& \sum_{i=1}^{l} \kappa_{i} \int_{\Sigma_{i}} \left[ \frac{R^{\Sigma_{i}}}{2} - \frac{n-2}{2n}(2\vert E\vert^{2} +2\Lambda) - \frac{2}{n} |E|^{2} \right]dA_g\nonumber\\
			&=& \sum_{i=1}^{l} \kappa_{i} \int_{\Sigma_{i}} \left[ \frac{R^{\Sigma_{i}}}{2}  -  |E|^{2} \right]dA_g -\frac{n-2}{n}\Lambda\sum_{i=1}^{l} \kappa_{i}|\Sigma_i|\nonumber\\
			&\leq& \sum_{i=1}^{l} \kappa_{i} \int_{\Sigma_{i}}  \frac{R^{\Sigma_{i}}}{2}dA_g - \frac{1}{2}(n-1)(n-2)\omega^{2}_{n-1}\sum_{i=1}^{l}  \frac{\kappa_{i}Q(\Sigma_i)^2}{\vert\Sigma_i\vert} -\frac{n-2}{n}\Lambda\sum_{i=1}^{l} \kappa_{i}|\Sigma_i|.\nonumber
		\end{eqnarray}
		which proves the stated inequality \eqref{Crusc_eq}.

		Finally, observe that, since both terms in the integral over \( M^n \) are nonnegative, equality holds if and only if
		$$\frac{2}{n-1}|E|^{2} \left(\Lambda - (n-2) |E|^{2} \right)=0\quad\mbox{and}\quad\overset{\circ}{T}=0.$$
From this, and using once again the assumption $\Lambda-(n-2)|E|^{2}>0$, we conclude that \( E \equiv 0 \) and, consequently, \( \overset{\circ}Ric = 0 \), that is, $M^n$ is an Einstein manifold. By applying the classical Reilly's theorem (see \cite[Lemma 3]{reilly2}), we obtain the rigidity result.

	\end{proof}

As a direct consequence of Theorem \ref{Crusc_th}, Theorem E of \cite{tiarlos} can be recovered by using the Gauss-Bonnet formula. To be precise, we have the following corollary.

\begin{corollary}(\cite[Theorem E]{tiarlos})\label{cororo}
		Let $(M^3,\, g,\, f,\, E)$ be a compact electrostatic system such that $E$ is parallel to $\nabla f$ and connected boundary $\partial M$. Suppose that $\vert E\vert^{2} < \Lambda$. Then we have:
		\begin{equation}\label{crusc_eq1}
			16\pi^2 \frac{Q(\partial M)^2}{\vert\partial M\vert} +\frac{1}{3}\Lambda|\partial M|\leq 4\pi ,
		\end{equation}
		Moreover, equality occurs in \eqref{crusc_eq1} if and only if $(M^3,\,g,\,f,\,E)$ is isometric to the de Sitter system.  
	\end{corollary}

In the higher-dimensional setting, the additional assumption that the boundary is Einstein yields the following boundary estimate.

\begin{theorem}\label{alan1}
Let $(M^n,\,g,\,f,\, E)$ be a compact electrostatic system with connected Einstein boundary $\partial M$ and such that $E$ is parallel to $\nabla f.$ Suppose that $\vert E\vert^{2}<\frac{\Lambda}{(n-2)}.$ Then we have: 
\begin{equation}\label{bound_est}
|\partial M| \left[\displaystyle\min_{\partial M}\left(R+\frac{4}{n-2}|E|^{2}\right)\right]^{\frac{n-1}{2}}\leq [n(n-1)]^{\frac{n-1}{2}}\omega_{n-1},
\end{equation}
where $\omega_{n-1}$ is the area of the standard unitary sphere $\mathbb{S}^{n-1}$. Moreover, equality holds in (\ref{bound_est}) if and only if $(M^n,\,g,\,f,\,E)$ is isometric to the de Sitter system.  
\end{theorem}

\begin{proof}
Since $\partial M$ is Einstein, we may consider $Ric^{\partial M}=\delta(n-2)g^{\partial M},$ where $\delta$ is constant. By Bishop's theorem, one deduces that
\begin{equation}\label{est1}
|\partial M|\leq\delta^{-\frac{1}{a}}\omega_{n-1}.
\end{equation}

On the other hand, since $R^{\partial M}=(n-1)(n-2)\delta$, Theorem \ref{Crusc_th} implies
		\begin{equation*}
			\delta(n-1)(n-2)\geq\displaystyle\min_{\partial M}\left(\frac{n-2}{n}R+\frac{4}{n}|E|^{2}\right),
		\end{equation*}
		so that
		\begin{equation}\label{est2}
			\delta n(n-1)\geq \min_{\partial M}\left(R+\frac{4}{n-2}|E|^{2}\right).
		\end{equation}
		Next, by assumption, we have 
		\[
		\min_{\partial M} \left( R + \frac{4}{n-2} |E|^{2} \right) > 0.
		\]  
		Thereby, substituting \eqref{est2} into \eqref{est1}, we arrive at \eqref{bound_est}. Moreover, the equality in \eqref{est2} guarantees the equality in \eqref{Crusc_eq} of Theorem \ref{Crusc_th} and therefore, the rigidity result follows.
	\end{proof}

When \( n = 5 \), so that the boundary \( \partial M \) is four-dimensional, the Gauss--Bonnet--Chern theorem reveals a deep connection between the topology and geometry of \( \partial M \). In particular, it expresses the Euler characteristic \( \chi(\partial M) \) through the following identity:

\begin{equation}\label{CGB}
8\pi^{2} \chi(\partial M) = \frac{1}{4} \int_{\partial M} |W|^{2} \, dA_g + \frac{1}{24} \int_{\partial M} \left(R^{\partial M}\right)^{2} \, dA_g - \frac{1}{2} \int_{\partial M} |\mathring{Ric}_{_{\partial M}} |^{2} \, dA_g,
\end{equation}
where \( W \), \(R^{\partial M} \) and \( \overset{\circ}{Ric}_{\partial M} \) denote the Weyl tensor, the scalar curvature and the traceless part of its Ricci tensor of \( \partial M \), respectively. This result originates from the seminal contributions of S.-S. Chern \cite{chern1944} in the 1940s, wherein he generalized the classical Gauss–Bonnet theorem to higher even dimensions via the framework of characteristic classes and differential forms.

\begin{corollary}\label{coro2alan}
Let $(M^5,\,g,\,f,\,E)$ be a compact electrostatic system with connected Einstein boundary $\partial M$ such that $E$ parallel to $\nabla f.$ Then we have:
		\begin{equation}
		\label{olkj7801}
			\left[\displaystyle\min_{\partial M}\left(\frac{3}{5}R+\frac{4}{5}|E|^{2}\right)\right]^{2}|\partial M|\leq 	192\pi^{2}\chi(\partial M).
		\end{equation} Moreover, equality holds in (\ref{olkj7801}) if and only if $(M^5,\,g,\,f)$ is isometric to the de Sitter system. 
\end{corollary}
	
\begin{proof}
Considering $n=5$ in Theorem \ref{Crusc_th} (see (\ref{alan})), one obtains that
		$$\int_{\partial M}\left[R^{\partial M}-\frac{3}{5}R-\frac{4}{5}|E|^{2}\right]dA_g\geq 0,$$
		which then implies
		$$\displaystyle\min_{\partial M}\left(\frac{3}{5}R+\frac{4}{5}|E|^{2}\right)|\partial M|\leq\int_{\partial M}R^{\partial M}dA_g.$$
		Now, we use the Holder's inequality to infer
		$$\left[\displaystyle\min_{\partial M}\left(\frac{3}{5}R+\frac{4}{5}|E|^{2}\right)\right]^{2}|\partial M|\leq \int_{\partial M}\left(R^{\partial M}\right)^{2}dA_g.$$
To conclude, it suffices to apply the Gauss–Bonnet–Chern formula \eqref{CGB}. Furthermore, the equality case follows from Theorem \ref{Crusc_th}.
\end{proof}

In the sequel, assuming a suitable upper bound for the scalar curvature, we obtain a boundary estimate for compact electrostatic systems with connected boundary. This estimate will be used in the proof of Theorem~\ref{livreE}. 

To this end, we first establish a lower bound for the scalar curvature in electrostatic systems, which will also play a key role in the proof of Theorem~\ref{geometricineq}.
\begin{lemma}\label{lower_scal}
Let $(M^{n},\,g,\,f,\,E)$ be a compact electrostatic system with positive cosmological constant $\Lambda$. Then the scalar curvature satisfies the following lower bound:
\[
R \geq \frac{(n-1)\left(3n - 4 - \sqrt{(n+2)(n-2)}\right)}{(n-1)^2 + (n-2)^2} \Lambda.
\]
\end{lemma}

\begin{proof}
Consider the function
\[
\ell(n) = \frac{(n-1)\left(3n - 4 - \sqrt{(n+2)(n-2)}\right)}{(n-1)^2 + (n-2)^2}.
\]
A straightforward computation shows that $\ell(n)$ is a decreasing function of $n$. Moreover,
\[
\ell(2) = 2,\,\,\ell(3) = 2 \left(1 - \frac{\sqrt{5}}{5} \right) \approx 1.10, \quad \text{and} \quad \lim_{n \to +\infty} \ell(n) = 1.
\]
In particular, for all $n \geq 2$, we have
\[
\ell(n)\Lambda \leq 2\Lambda.
\]
Next, it follows from (\ref{rrr}) that $R \geq 2\Lambda$ and hence, the claimed inequality follows.
\end{proof}

We are now in a position to discuss our next result. Essentially, it provides a condition under which the left-hand side of inequality~\eqref{geoinq} in Proposition~\ref{propprop} is nonnegative. See also Remark~\ref{about_limit} for a discussion on the meaning of the hypothesis.

\begin{theorem}\label{geometricineq}
	Let $(M^{n},\,g,\,f,\,E)$ be a compact electrostatic system with connected boundary $\partial M$ and positive cosmological constant satisfying
	\begin{eqnarray}\label{upper_bound}
		R \leq  \frac{(n-1)\left({3n-4} + \sqrt{(n+2)(n-2)}\right)}{(n-1)^2+(n-2)^2}\Lambda.
	\end{eqnarray}
	Then we have:
	\begin{eqnarray}
		\label{km67401m}
		\kappa C\vert\partial M\vert \leq \frac{2\Lambda}{(n-1)}\int_{\partial M} R^{\partial M}dA_g,
	\end{eqnarray}
	where $\kappa = \vert\nabla f\vert\Big|_{\partial M}$ and $C\geq 0$ is a constant. Moreover, equality holds in (\ref{km67401m}) if and only if $(M^{n},\,g,\,f,\,E)$ is isometric to the de Sitter system with $C=\dfrac{4(n-2)\Lambda^2}{\kappa n(n - 1)}$ and $\kappa = \dfrac{\Lambda}{n}$.
\end{theorem}

\begin{proof} We start by defining
\begin{eqnarray}\label{bru}
\overline{F}(\vert E \vert^{2}) &=& \frac{4(n-1)}{n \kappa} \left\{ \vert E \vert^{2} + \frac{(n-2)^2}{4(n - 1)^2} \left[ \sqrt{R \left( R + \frac{8(n-1)^2}{(n-2)^3} \Lambda \right)} + R \right] \right\} \nonumber \\
&& \times \left\{ \frac{(n-2)^2}{4(n - 1)^2} \left[ \sqrt{R \left( R + \frac{8(n-1)^2}{(n-2)^3} \Lambda \right)} - R \right] - \vert E \vert^{2} \right\},
\end{eqnarray}
where \(\kappa = \left| \nabla f \right| \big|_{\partial M}\). Note that \(\overline{F}(\vert E \vert^{2})\) is a quadratic polynomial in the variable \(\vert E \vert^{2}\). From now on, the proof is divided into some steps.

We first need to obtain a geometric condition to guarantee that $\overline{F}(\vert E \vert^{2})\geq 0$. Indeed, since \(R \geq 0\), it suffices to control the sign of the second factor in the product in~\eqref{bru}. Using~\eqref{rrr}, we rewrite this factor in terms of the scalar curvature \(R\) as
\begin{eqnarray*}
P(R) &=& \frac{(n-2)^2}{4(n - 1)^2} \left[ \sqrt{R \left( R + \frac{8(n-1)^2}{(n-2)^3} \Lambda \right)} - R \right] - \vert E \vert^{2} \\
&=& \frac{(n-2)^2}{4(n - 1)^2} \sqrt{R \left( R + \frac{8(n-1)^2}{(n-2)^3} \Lambda \right)} - \frac{(n-2)^2}{4(n - 1)^2} R - \frac{R}{2} + \Lambda.
\end{eqnarray*} By squaring this expression, we get

\begin{eqnarray}
\left[\frac{(n-2)^2}{4(n - 1)^2}\right]^{2} R \left( R + \frac{8(n-1)^2}{(n-2)^3} \Lambda \right) 
&=& P(R)^{2} + 2P(R) \left( \frac{(n-2)^2}{4(n - 1)^2} R + \frac{R}{2} - \Lambda \right) \nonumber \\
&& + \left( \frac{(n-2)^2}{4(n - 1)^2} R + \frac{R}{2} - \Lambda \right)^{2}. \label{sqr_P}
\end{eqnarray} Observe from (\ref{rrr}) that \(R \geq 2|E|^{2}\) and therefore,
\[
\left[\frac{(n-2)^2}{4(n - 1)^2} R + \frac{R}{2} - \Lambda \right] \geq \left[\frac{(n-2)^2}{2(n - 1)^2}|E|^{2} + |E|^{2} \right] \geq 0.
\]
Furthermore, to guarantee that \(P(R) \geq 0\) we only need to show that
\begin{eqnarray}\label{cuidado com o sinal}
Q(R) := -\frac{(n-1)^2 + (n-2)^2}{4(n-1)^2} R^{2} + \frac{3n - 4}{2(n - 1)} \Lambda R - \Lambda^{2} \geq 0.
\end{eqnarray} The roots of \(Q(R)\) are
\[
R_{\pm} = \frac{(n-1) \left( 3n - 4 \pm \sqrt{(n+2)(n-2)} \right)}{(n-1)^2 + (n-2)^2} \Lambda.
\] Consequently, by Lemma~\ref{lower_scal} and the assumption~\eqref{upper_bound}, one sees that
\[
R_{-} \leq R \leq R_{+},
\]
and therefore \(Q(R)\), \(P(R)\) and $\overline{F}(\vert E\vert^{2})$ are nonnegative.

Proceeding, we return to inequality \eqref{geoinq} in Proposition \ref{propprop} which asserts that 
\begin{eqnarray*}
&&	\frac{1}{\kappa}\int_{M}\left[\frac{1}{f}\vert\mathring{\nabla}^2f\vert^{2} + f\vert\mathring{R}ic\vert^{2}\right]dV_{g}\nonumber\\
&& + \frac{4(n-1)}{n\kappa}\int_{M}f\left\{\vert E\vert^{2} +  \dfrac{(n-2)^2}{4(n - 1)^2}\left[\sqrt{R\left(R + \frac{8(n-1)^2}{(n-2)^3}\Lambda\right)} +R\right]\right\}\nonumber\\
&&\times\left\{\dfrac{(n-2)^2}{4(n - 1)^2}\left[\sqrt{R\left(R + \frac{8(n-1)^2}{(n-2)^3}\Lambda\right)} - R\right] - \vert E\vert^{2}\right\}dV_{g}\nonumber\\
&=& \int_{\partial M} R^{\partial M} dA_{g},
\end{eqnarray*}
so that
	\begin{eqnarray*}
		\frac{1}{\kappa}\int_{M}\left[\frac{1}{f}\vert\mathring{\nabla}^2f\vert^{2} + f\vert\mathring{R}ic\vert^{2}\right]dV_{g} + \int_{M}f\overline{F}(\vert E\vert^{2})dV_{g}=
		\int_{\partial M} R^{\partial M} dA_{g}.
	\end{eqnarray*}
By letting \(C := \min_M \overline{F}(\vert E\vert^{2}) \geq 0,\) one obtains that
\begin{eqnarray}\label{ineq_FE}
C\int_{M}fdV_{g}\leq
\int_{\partial M} R^{\partial M} dA_{g}.
\end{eqnarray} Moreover, equality in \eqref{ineq_FE} occurs if and only if $\mathring{R}ic=0$, and $\mathring{\nabla}^{2}f=0$.

We now need to remove the dependence on $f.$ To that end, on integrating the equation in (\ref{s1}):
\[
\Delta f = \frac{2}{n-1} \left[(n-2)|E|^{2} - \Lambda\right] f
\]
and applying Proposition~\ref{properties}, we obtain
\begin{eqnarray*}
	- \kappa \vert \partial M \vert \geq \frac{2}{n-1} \left((n-2) \min_M |E|^{2} - \Lambda\right) \int_M f \, dV_g.
\end{eqnarray*}
Since $C \geq 0$, one deduces that
\begin{eqnarray*}
	\kappa C \vert \partial M \vert &\leq& \frac{2C}{n-1} \left( \Lambda - (n-2) \min_M |E|^{2} \right) \int_M f \, dV_g \\
	&\leq& \frac{2C\Lambda}{n-1} \int_M f \, dV_g \\
	&\leq& \frac{2\Lambda}{n-1} \int_{\partial M} R^{\partial M} \, dA_g,
\end{eqnarray*}
so that
\begin{eqnarray} \label{mainteo4}
	\kappa C \vert \partial M \vert \leq \frac{2\Lambda}{n-1} \int_{\partial M} R^{\partial M} \, dA_g,
\end{eqnarray} which proves the asserted inequality.

Finally, we address the equality case. If equality is achieved, then \eqref{ineq_FE} yields $\mathring{\operatorname{Ric}} = 0$ and $\mathring{\nabla}^{2}f = 0$. By Reilly's theorem \cite{reilly2}, it follows that the manifold is the de Sitter space and then $\kappa = \dfrac{\Lambda}{n}$.	Furthermore, we can explicitly compute the constant $C$ by evaluating
\begin{eqnarray*}
	\overline{F}(0) &=& \frac{4(n-1)}{n\kappa}\left\{  \dfrac{(n-2)^2}{4(n - 1)^2}\left[\sqrt{2\Lambda\left(2\Lambda + \frac{8(n-1)^2}{(n-2)^3}\Lambda\right)} + 2\Lambda\right]\right\} \\
	&& \times\left\{\dfrac{(n-2)^2}{4(n - 1)^2}\left[\sqrt{2\Lambda\left(2\Lambda + \frac{8(n-1)^2}{(n-2)^3}\Lambda\right)} - 2\Lambda\right] \right\} \\
	&=& \dfrac{(n-2)^4\Lambda^2}{\kappa n(n - 1)^3} \left[\sqrt{1 + \frac{4(n-1)^2}{(n-2)^3}} + 1\right]\left[\sqrt{1 + \frac{4(n-1)^2}{(n-2)^3}} - 1\right] \\
	&=& \dfrac{4(n-2)\Lambda^2}{\kappa n(n - 1)}.
\end{eqnarray*} Thus, we must have $C = \dfrac{4(n-2)\Lambda^2}{\kappa n(n - 1)}$. So, the proof is finished. 

\end{proof}

\begin{remark}
\label{about_limit}
	Regarding the scalar curvature restriction in the assumption of Theorem~\ref{geometricineq}, notice that
	\[
	\lim_{n \to +\infty} \frac{(n-1) \left( 3n - 4 + \sqrt{(n+2)(n-2)} \right)}{(n-1)^2 + (n-2)^2} = 2.
	\]
	
	Moreover, for \( n = 3 \), we have the explicit bound
	\[
	1,10\,\Lambda\approx	2 \left( 1 - \frac{\sqrt{5}}{5} \right) \Lambda \leq R \leq 2 \left( 1 + \frac{\sqrt{5}}{5} \right) \Lambda \approx 2,89\,\Lambda.
	\]
	Note that this hypothesis is consistent with the restriction observed in Proposition~\ref{properties}.
\end{remark}

We proceed to prove Theorem~\ref{livreE}, which follows as a direct consequence of Theorem~\ref{geometricineq}.

\begin{proof}[{\bf Proof of Theorem \ref{livreE}}]
We begin by using (\ref{rrr}) to verify that the assumption in Theorem~\ref{geometricineq} is satisfied. Therefore, it suffices to apply the Gauss-Bonnet formula to the right-hand side of equation~(\ref{mainteo4}) in order to obtain the stated inequality.  
\end{proof}

It is also relevant to consider the case where the surface gravity does not appear explicitly. In this context, we have the following theorem.

\begin{theorem}
\label{teorigidezint}
		Let $(M^n,\, g,\, f,\, E)$ be a compact electrostatic system and $\partial M=\cup_{i=1}^{l} \Sigma_{i}$, where $\Sigma_{i}$ are the connected components of $\partial M$. Suppose that $$[(n-1)^2+(n-2)^2]\vert E\vert^{2}\leq (n-2)\Lambda.$$ Then we have:
		\begin{equation*}
			\frac{n-2}{n}\Lambda\sum_{i=1}^{l} \kappa_{i}|\Sigma_i|\leq\sum_{i=1}^{l} \kappa_{i} \int_{\Sigma_{i}}  \frac{R^{\Sigma_{i}}}{2}dA_g ,
		\end{equation*}
where $\kappa_{i}=|\nabla f|\Big|_{\Sigma_{i}}$ and $\omega_{n-1}$ is the area of the standard $(n-1)$-sphere. Moreover, equality holds if and only if $(M^n,\,g,\,f,\,E)$ is isometric to the de Sitter system.
\end{theorem}

\begin{proof}
By Proposition \ref{propprop}, we have 

		\begin{eqnarray*}
			\int_{M}\left[\frac{1}{f}\vert\mathring{\nabla}^2f\vert^{2} + f\vert\mathring{R}ic\vert^{2}\right] dV_{g}	&=&	\int_{\partial M}\vert\nabla f\vert R^{\partial M} dA_{g} \nonumber\\
			&&+ \frac{1}{n}\int_{M}\left(4(n-1)f\vert E\vert^{4}+(n-2)R\Delta f\right) dV_{g},
		\end{eqnarray*} by using \eqref{rrr}, one obtains that
		\begin{eqnarray*}
			n\int_{M}\left[\frac{1}{f}\vert\mathring{\nabla}^2f\vert^{2} + f\vert\mathring{R}ic\vert^{2}\right]dV_{g}	&=&	n\kappa\int_{\partial M} R^{\partial M} dA_{g} \nonumber\\
			&&+\int_{M}\left(4(n-1)f\vert E\vert^{4}+2(n-2)(\vert E\vert^2 + \Lambda)\Delta f\right) dV_{g}\nonumber\\
			&=&	n\kappa\int_{\partial M} R^{\partial M} dA_{g} \nonumber\\
			&&+\int_{M}2\vert E\vert^{2}\left( 2(n-1)f\vert E\vert^{2}+(n-2)\Delta f\right) dV_{g} \nonumber\\
			&&- 2(n-2)\Lambda\kappa\vert\partial M\vert
		\end{eqnarray*}
		where $\kappa = \vert\nabla f\vert\Big|_{\partial M}$. Then, by Definition \ref{def1}, we get
		\begin{eqnarray*}
			&&	n\int_{M}\left[\frac{1}{f}\vert\mathring{\nabla}^2f\vert^{2} + f\vert\mathring{R}ic\vert^{2}\right]dV_{g}
			+  2(n-2)\Lambda\kappa\vert\partial M\vert=	n\kappa\int_{\partial M} R^{\partial M} dA_{g} \nonumber\\
			&&+\frac{4}{(n-1)}\int_{M}f\vert E\vert^{2}\left[((n-1)^{2}+(n-2)^{2})\vert E\vert^{2}  - 
			(n-2)\Lambda\right]dV_{g}.
		\end{eqnarray*} Therefore, by assuming that $[(n-1)^2+(n-2)^2]\vert E\vert^{2}\leq (n-2)\Lambda,$ one concludes that
		\begin{eqnarray*}
			\Lambda\vert\partial M\vert\leq \frac{n}{2(n-2)}\int_{\partial M} R^{\partial M} dA_{g}.
		\end{eqnarray*} Furthermore, equality holds if and only if $\vert\mathring{\nabla}^2f\vert^{2}=0$, $\vert\mathring{R}ic\vert^{2}=0$ and $E=0$. By applying the classical Reilly's theorem, we obtain the rigidity result.
	\end{proof}

\begin{proof}[{\bf Proof of Theorem \ref{teorigidezint1}}]
The proof follows immediately from Theorem~\ref{teorigidezint}, in combination with the Gauss–Bonnet formula.
\end{proof}

In the remainder of this section, we present an additional boundary area estimate in higher dimensions under the  assumption of constant scalar curvature. In this case, the technique differs slightly from the previous ones and is based on a volume comparison argument involving the Bishop–Gromov inequality. We note that it would be of particular interest to replace the constant scalar curvature assumption by a condition involving a bound on $|E|$. 

\begin{theorem}
\label{thm6}
Let \( (M^n,\,g,\,f,\,E) \) be a compact electrostatic system with connected boundary \( \partial M \) and constant scalar curvature. Then we have:
\begin{equation}
\label{lkm1111}
|\partial M| \leq c\, \Omega_{n-1},
\end{equation}
for a specific constant \( c > 0 \), where \( \Omega_{n-1} \) denotes the volume of the unit \((n-1)\)-dimensional sphere. Moreover, equality holds in (\ref{lkm1111}) if and only if \( (M^n,\,g,\,f,\,E) \) is isometric to the de Sitter system.
\end{theorem}

\begin{remark}
The constant \( c \) in Theorem \ref{thm6} is given by
\begin{equation*}
c := \frac{1}{\kappa} \cdot \frac{n^{\frac{n-1}{2}}}{\zeta^{\frac{n+1}{2}}},
\end{equation*}
where \( \zeta = \frac{2}{n-1} \left[ \Lambda - (n-2)|E|^2 \right] \).
\end{remark}

\begin{proof}
Initially, it follows from the system \eqref{s1} that
\[
- \frac{\Delta f}{f} = \frac{2}{n-1} \left[ \Lambda - (n-2)|E|^2 \right] =: \zeta,
\]
where \( \zeta > 0 \) is a constant by the constant scalar curvature assumption and (\ref{rrr}). Next, consider the warped product \( N^{n+1} = \mathbb{S}^1 \times_f M \), and denote geometric tensors on \( M \) and \( N \) with the subscripts \( M \) and \( N \), respectively. Using standard curvature formulas for warped products (see \cite[Chapter~7]{oneill}), one sees that
\begin{align*}
Ric_N(X,Y) &= Ric_M(X,Y) - \frac{1}{f} \nabla^2 f(X,Y), \\
Ric_N(X,V) &= 0, \\
Ric_N(U,V) &= -\frac{\Delta f}{f} \, g_{\mathbb{S}^1}(U, V),
\end{align*}
for vector fields \( X, Y \in \mathfrak{X}(M) \) and \( U, V \in \mathfrak{X}(\mathbb{S}^1) \).
Hence, one deduce from \eqref{sub_static} in Proposition~\ref{properties} that
\begin{equation}\label{lim_Ricci}
Ric_M - \frac{1}{f} \nabla^2 f \geq -\frac{\Delta f}{f} \, g_M,
\end{equation}
and consequently,
\[
Ric_N \geq \zeta\, g_N.
\]
Applying the Bishop-Gromov comparison theorem, we obtain

\begin{equation}
\label{olk780AA}
Vol(N) \leq \Omega_{n+1} \left( \frac{n}{\zeta} \right)^{\frac{n+1}{2}}.
\end{equation}

On the other hand, by Fubini's theorem and integration by parts, we have
\begin{eqnarray}
\label{olk780}
Vol(N) &=& \int_{\mathbb{S}^1} \left[ \int_M f \, dV_g \right] d\theta = 2\pi \int_M f \, dV_g \nonumber\\
&=& -\frac{2\pi}{\zeta} \int_M \Delta f \, dV_g = \frac{2\pi}{\zeta} \int_{\partial M} |\nabla f| \, dA_g \nonumber\\
&=& \frac{2\pi}{\zeta} \kappa \, |\partial M|,
\end{eqnarray}
where \( \kappa = |\nabla f|_{\partial M} \) is constant. Plugging (\ref{olk780}) into (\ref{olk780AA}), we get
\[
|\partial M| \leq \frac{\Omega_{n+1}}{2\pi \kappa} \left( \frac{n}{\zeta} \right)^{\frac{n+1}{2}}.
\]
Using the recurrence formula \( \Omega_{n-1} = \frac{n}{2\pi} \Omega_{n+1} \), one deduces that
\[
|\partial M| \leq \frac{1}{\kappa} \cdot \frac{n^{\frac{n-1}{2}}}{\zeta^{\frac{n+1}{2}}} \Omega_{n-1}.
\]

In the case of equality, Bishop’s theorem implies that 
\( (N^{n+1}, h) \) must be isometric to the round sphere. The equality in \eqref{lim_Ricci} then yields \( Ric_M - \frac{1}{f} \nabla^2 f = -\frac{\Delta f}{f} g_M \), which holds if and only if \( E = 0 \), meaning the system is vacuum static. Furthermore, by the classical characterization of Einstein manifolds associated with vacuum static spaces (see, for instance, \cite[Section~5]{ambrozio}), we conclude that \( (M, g) \) is isometric to the de Sitter system, which finishes the proof. 
\end{proof}

\section{Brown-York and charged Hawking masses}

In this section, we present the proofs of Theorems~\ref{teoBY} and~\ref{hawking2}. It is worth noting that the technique adopted here is inspired by recent works in various contexts of Einstein-type spaces, including the vacuum static case \cite{FY,HMR,Jie,yuan}, perfect fluid space-times \cite{costa} and quasi-Einstein manifolds \cite{ernani}. We begin by establishing the bound for the Brown–York mass in electrostatic systems. To do so, we consider the following notation

$$\beta^{-1}=\displaystyle\max_{M}\left(f^{2}+\frac{n(n-1)}{R}|\nabla f|^{2}\right)^{\frac{1}{2}}.$$ In addition, we introduce the function $v$ along with a conformal change of the metric $\overline{g}$ by
	$$v=(1+\beta f)^{-\frac{n-2}{2}}\,\,\quad\mbox{and}\quad\,\,\overline{g}=v^{\frac{4}{n-2}}g.$$ With this notation, we have the following lemma.

	\begin{lemma}\label{posit_conf}
Let $(M^n,\,g,\,f,\,E)$ be a compact electrostatic system with $\Lambda+|E|^{2}>0$. Then $R_{\bar{g}}$ is non-negative. Furthermore, $\overline{R}=0$ if and only if $E\equiv 0$, $\Delta f=-\frac{R}{n-1}f$ and $f^{2}+\frac{n(n-1)}{R}|\nabla f|^{2}$ is constant in $M^n$. Here, $\overline{R}$ stands for the scalar curvature of $\overline{g}$.
	\end{lemma}

	\begin{proof}
		First, observe that
		\begin{eqnarray}
			\Delta f&=&\left(2\left(\frac{n-2}{n-1}\right)|E|^{2}-\frac{2}{(n-1)}\Lambda\right)f\nonumber\\
			&=&\left(2\left(\frac{n-2}{n-1}\right)|E|^{2}+\frac{2}{n-1}|E|^{2}-\frac{R}{n-1}\right)f\nonumber\\
			&=&\left(2|E|^{2}-\frac{R}{n-1}\right)f\nonumber\\
			&\geq & -\frac{R}{n-1}f, \label{ineq_el1}
		\end{eqnarray}
		where equality holds if and only if $E=0.$ From this, it follows that
		\begin{eqnarray}\nonumber
			\Delta v &=&\,\Delta\left((1+\beta f)^{-\frac{n-2}{2}}\right)\\\nonumber
			&=&\,-\frac{(n-2)\beta}{2}div\left((1+\beta f)^{-\frac{n}{2}}\nabla f\right)\\
			&=&\, -\frac{(n-2)\beta}{2}(1+\beta f)^{-\frac{n}{2}}\Delta f + \frac{n(n-2)}{4}\beta^2 (1+\beta f)^{-\frac{n+2}{2}}|\nabla f|^2\nonumber\\
			&\leq &\frac{n(n-2)}{4}\beta(1+\beta f)^{-\frac{n+2}{2}}\left(\frac{2R}{n(n-1)}(1+\beta f)f+\beta|\nabla f|^2\right). \label{ineq_el2}
		\end{eqnarray}

		On the other hand, we use the following formula for the scalar curvature 
		\begin{eqnarray*}
			{\overline{R}}=v^{-\frac{n+2}{n-2}}\left(R v-4\frac{(n-1)}{(n-2)}\Delta v\right),
		\end{eqnarray*} 
		in order to infer
		\begin{eqnarray*}
			\Delta v&=&\;\frac{(n-2)}{4(n-1)}R v-\frac{(n-2)}{4(n-1)}v^{\frac{n+2}{n-2}}{\overline{R}}\\
			&=&\;\frac{(n-2)}{4(n-1)}R (1+\beta f)^{-\frac{n-2}{2}}-\frac{(n-2)}{4(n-1)}(1+\beta f)^{-\frac{n+2}{2}}{\overline{R}}\\
			&=&\;\frac{(n-2)}{4(n-1)}(1+\beta f)^{-\frac{n+2}{2}}\left[(1+\beta f)^2R-{\overline{R}}\right].
		\end{eqnarray*}
		Substituting this information into \eqref{ineq_el2} yields
		\begin{eqnarray*}
			R(1+\beta f)^2- {\overline{R}}\leq n(n-1)\beta \left(\frac{2 R}{n(n-1)}(1+\beta f)f+\beta|\nabla f|^2\right).
		\end{eqnarray*} 
		Consequently,
		\begin{eqnarray}\nonumber
			{\overline{R}}&\geq &\, R(1+\beta f)^2-n(n-1)\beta \left(\frac{2 R}{n(n-1)}(1+\beta f)f+\beta|\nabla f|^2\right)\\\nonumber
			&=&\, R\left[(1+\beta f)^2-2\beta(1+\beta f)f-\frac{n(n-1)\beta^2}{R}|\nabla f|^2\right]\\\label{eqq3}
			&=&\, R\left[1-\beta^2\left(f^2+\frac{n(n-1)}{R}|\nabla f|^2\right)\right].
		\end{eqnarray} Since $R>0$ (whenever $\Lambda+|E|^{2}>0$), and using the chosen value of $\beta$, we conclude that ${\overline{R}}\geq 0.$
		
		Finally, it follows from \eqref{ineq_el1} that ${\overline{R}}=0$ if and only if $E\equiv 0$ and $\Delta f=-\frac{R}{n-1}f$.  Moreover, from the characterization of $\beta$, we deduce that the quantity $f^2+\frac{n(n-1)}{R}|\nabla f|^2$ is constant on $M^n,$ which finishes the proof. 
\end{proof}

	We are now ready to present the proof of Theorem~\ref{teoBY}.

	\begin{proof}[{\bf Proof of Theorem \ref{teoBY}}]
Following the notation introduced above, it is well known that the mean curvature $\overline{H}^{i}$ of $\Sigma_i$ with respect to the conformal metric $\overline{g}=v^{\frac{4}{n-2}}g$ is strictly positive. Indeed, since $f|_{\partial M}=0,$ it follows that $v\mid_{\partial M}=1.$ Hence, $\overline{g}=g$ over the boundary $\partial M$ and $(\Sigma_i,\overline{g})$ is isometric to $(\Sigma_{i},\,g)$, which, by assumption, can be isometrically embedded in $\mathbb{R}^n$ as a convex hypersurface with mean curvature  $H_{0}^{i},$ induced by the Euclidean metric. Moreover, taking into account that the mean curvature of $\Sigma_i$ with respect to the conformal metric $\overline{g}$ is given by
		\begin{eqnarray}\label{eqq5}
			\overline{H}^{i}=v^{-\frac{2}{n-2}}\left(H^{i}+2\frac{n-1}{n-2}\partial_\nu(\log(v))\right),
		\end{eqnarray} so that
		
		\begin{eqnarray}
			\label{eqq6}
			\overline{H}^{i}&=& \frac{2(n-1)}{(n-2)}\big\langle \nabla(1+\beta f)^{-\frac{n-2}{2}},\,\nu\big\rangle\nonumber\\
			&=& (n-1)\beta|\nabla f|\Big|_{_{\Sigma_i}},
		\end{eqnarray} where we have used that $H^{i}=0$ and $\nu=-\frac{\nabla f}{|\nabla f|}.$ This proves that $\overline{H}^{i}>0,$ as claimed (cf. Lemma 4 in \cite{tiarlos}).

		Proceeding, one sees from \eqref{eqq6} that
		
		\begin{eqnarray}
			\label{eqq7}
			\mathfrak{M}_{BY}(\Sigma_i,\overline{g})&=& \int_{\Sigma_i}(H_{0}^{i}-\overline{H}^{i}) dA_g\nonumber\\
			&=& \mathfrak{M}_{BY}(\Sigma_i,g)-(n-1)\beta|\nabla f|\Big|_{{\Sigma_i}}|\Sigma_i|.
		\end{eqnarray} Since Lemma \ref{posit_conf} implies that ${\overline{R}}\geq 0$,  and $\overline{H}^{i}>0,$ we apply the Positive Mass Theorem for the Brown-York mass \cite{Shi-Tam} to conclude that $\mathfrak{M}_{BY}(\Sigma_i,\overline{g})\geq 0.$ Consequently, 
		\begin{eqnarray}
			\label{eqq8}
			|\Sigma_i|&\leq &\frac{1}{(n-1)\beta|\nabla f|\big|_{{\Sigma_i}}}\mathfrak{M}_{BY}(\Sigma_i,g)\nonumber\\
			&=&\frac{1}{(n-1)\beta|\nabla f|\big|_{{\Sigma_i}}}\int_{\Sigma_i}H_{0}^{i} dA_g,
		\end{eqnarray} which proves the stated inequality \eqref{eqq4}.

		We now consider the case of equality. If equality holds in \eqref{eqq8} for some component $\Sigma_{i_{0}}$, then 		
		\begin{eqnarray*}
			\mathfrak{M}_{BY}(\Sigma_{i_{0}},\overline{g})=0.
		\end{eqnarray*} By invoking the equality case of the Positive Mass Theorem for the Brown–York mass, it follows that the conformal metric $\overline{g}$ is flat, and consequently, $(M^n,\,\overline{g})$ is isometric to a bounded domain in $\mathbb{R}^n$. Moreover, taking into account that ${\overline{R}}=0,$ Lemma~\ref{posit_conf} implies that $E\equiv 0$, $\Delta f=-\dfrac{R}{n-1}f$ and the quantity $\dfrac{n(n-1)}{R}|\nabla f|^2+f^2$ is constant on $M^n.$ Thus, since $(M^{n},\,g)$ is vacuum static, the scalar curvature $R$ is constant. Consequently, 
		\begin{eqnarray*}
			0=\nabla\left[\frac{n(n-1)}{R}|\nabla f|^2+f^2\right]=\frac{2n(n-1)}{R}\nabla^2 f(\nabla f)+2f\nabla f
		\end{eqnarray*}
		so that
		\begin{eqnarray*}
			\nabla|\nabla f|^2-2\frac{\Delta f}{n}\nabla f=0.
		\end{eqnarray*} Now, it suffices to use the Robinson-Shen-type identity in Lemma \ref{lemmaAk} to conclude that $|\mathring{Ric}|^2=0$. Finally, we use Reilly's theorem linked with the fact that $\partial M$ is totally geodesic in order to conclude that $(M^n,\,g)$ is isometric to $\mathbb{S}^{n}_{+}$.

		On the other hand, if $(M^n,\,g)$ is isometric to $\mathbb{S}^{n}_{+},$ one deduces that $\partial M=\mathbb{S}^{n-1}.$ Thereby, the Brown-York mass of $\mathbb{S}^{n-1}$ is given by
		\begin{eqnarray*}
			\mathfrak{M}_{BY}(\mathbb{S}^{n-1})=\int_{\mathbb{S}^{n-1}}(n-1) dA_{g_{\mathbb{S}^{n-1}}}=(n-1)\omega_{n-1},
		\end{eqnarray*}  where $\Omega_{n-1}$ is the volume of the standard $(n-1)$-sphere. At the same time, since $\mathbb{S}_+^n$ has constant scalar curvature $R=n(n-1),$ it follows that the height function $h$ satisfies $\Delta h=-\frac{R}{n-1}h=-nh$. Therefore, $(\mathbb{S}_+^n,\,g,\,h)$ is a static vacuum space with non-zero cosmological constant. Consequently, $h^2+\dfrac{n(n-1)}{R}|\nabla h|^2$ is constant. Indeed, a direct computation yields
		\begin{eqnarray*}
			\nabla\left(h^2+\frac{n(n-1)}{R}|\nabla h|^2\right)&=&2h\nabla h+2\nabla^2 h(\nabla h)\\
			&=&-\frac{2\Delta h}{n}\nabla h+2\nabla^2 h(\nabla h)\\
			&=&2\mathring{\nabla}^2 h(\nabla h)=2h\mathring{Ric}(\nabla h)=0.
		\end{eqnarray*} From this, it follows that 
		\begin{eqnarray*}
			\beta^{-2}=\left.\left(h^2+\frac{n(n-1)}{R}|\nabla h|^2\right)\right|_{\partial M}=|\nabla h|^{2}_{|{\partial M}},
		\end{eqnarray*} so that 
		
		$$\beta=\frac{1}{|\nabla h|_{|{\partial M}}}.$$ Hence,
		\begin{eqnarray*}
			\frac{\mathfrak{M}_{BY}(\partial M,g)}{\alpha (n-1)|\nabla h|_{|_{\partial M}}}=\omega_{n-1}=|\partial M|,
		\end{eqnarray*} which is the equality in \eqref{eqq8}. So, the proof is completed.
	\end{proof}

\begin{remark}
We note that Theorem~\ref{teoBY} can also be formulated within a broader class of sub-static systems (see Remark~\ref{sub_sta}). The core of its proof relies on inequality~\eqref{ineq_el1}, which remains valid in this more general setting. This extension has already been observed, for example, in the case where the energy–momentum tensor corresponds to that of a perfect fluid, under a sub-static condition related to the dominant energy condition, specifically, the inequality $\rho + P \geq 0$, where $\rho$ denotes the energy density and $P$ is the pressure; see \cite[Theorem~2]{costa}. This broader perspective allows the result to be extended to other classes of energy–momentum tensors, such as those associated with wave maps, including Klein–Gordon-type matter fields (cf. \cite[Section 2]{cfmr}).
\end{remark}

	We now recall that the charged Hawking mass of a $2$-surface $\Sigma$ is given by
	\begin{equation*}
		\mathfrak{M}_{CH}(\Sigma) = \sqrt{\frac{|\Sigma|}{16 \pi}}\left(\frac{1}{2}\chi(\Sigma) - \frac{1}{16\pi}\int_{\Sigma}(H^2 + \frac{4}{3}\Lambda)dA_{g} + \frac{4\pi}{|\Sigma|}Q(\Sigma)^2 \right),
	\end{equation*} where $|\Sigma|$ stands for the area of $\Sigma,$ $\chi(\Sigma)=2\big(1-g(\Sigma)\big)$ is the Euler characteristic and $g(\Sigma)$ stands for the genus of $\Sigma,$ respectively. We point out that, even in the case of minimal surfaces, the positivity of the charged Hawking mass is not trivial to establish. However, a straightforward computation shows that the charged Hawking mass of a sphere in the Reissner-Nordstr\"om-de Sitter space satisfies $\mathfrak{M}_{CH}(\mathbb{S}(r))=\mathfrak{M}_{ADM}$ (see \cite[p. 6]{baltazar2023}).

\begin{proof}[{\bf Proof of Theorem \ref{hawking2}}]
To begin with, we notice that a direct computation by using the definition of charged Hawking mass and Eq. \eqref{crusc_eq1} yields

\begin{eqnarray*}
			\mathfrak{M}_{CH}(\partial M)&=&\sqrt{\frac{|\partial M|}{16 \pi}}\left(\frac{1}{2}\chi(\partial M) - \frac{1}{16\pi}\int_{\partial M}(H^2 + \frac{4}{3}\Lambda)dA_g + \frac{4\pi}{|\partial M|}Q(\partial M)^2 \right)\\
			&=&\sqrt{\frac{|\partial M|}{16 \pi}}\left(1 - \frac{\Lambda}{12\pi}\vert\partial M\vert + \frac{4\pi}{|\partial M|}Q(\partial M)^2 \right)\\
			&\geq&\left( \frac{8\pi}{|\partial M|}Q(\partial M)^2 \right)\sqrt{\frac{|\partial M|}{16 \pi}},
		\end{eqnarray*}  which proves the first assertion. Moreover, the rigidity statement follows from Theorem \ref{cororo}.

Proceeding, we shall prove the second inequality. Indeed, by combining the definition of charged Hawking mass and Theorem \ref{teorigidezint1}, one obtains that
		\begin{eqnarray*}
			\mathfrak{M}_{CH}(\partial M)&=&\sqrt{\frac{|\partial M|}{16 \pi}}\left(\frac{1}{2}\chi(\partial M) - \frac{1}{16\pi}\int_{\partial M}(H^2 + \frac{4}{3}\Lambda)dA_g + \frac{4\pi}{|\partial M|}Q(\partial M)^2 \right)\\
			&=&\sqrt{\frac{|\partial M|}{16 \pi}}\left(1 - \frac{\Lambda}{12\pi}\vert\partial M\vert + \frac{4\pi}{|\partial M|}Q(\partial M)^2 \right)\\
			&\geq&\left( \frac{4\pi}{|\partial M|}Q(\partial M)^2 \right)\sqrt{\frac{|\partial M|}{16 \pi}},
		\end{eqnarray*} which gives the stated inequality. Furthermore, the rigidity statement follows from Theorem  \ref{teorigidezint1}. This finishes the proof of the theorem. 
		
	\end{proof}

	\noindent{\bf Conflict of Interest:} There is no conflict of interest to disclose.
	
	\

\noindent{\bf Data Availability:} Not applicable.
	
	\
	
\noindent{{\bf Acknowledgments.}} The authors would like to thank the referees for their careful reading and valuable suggestions. Allan Freitas wishes to thank IMECC/UNICAMP for the stimulating scientific environment provided during the development of part of this work. He is particularly grateful to Lino Grama for his warm hospitality and constant encouragement.

\end{document}